\documentclass{article}

\usepackage[margin=1.3in]{geometry}

\usepackage{amsmath,amssymb,amsthm,mathrsfs,esint,color,times,textcomp,sectsty,verbatim,graphicx,enumerate,authblk,amsxtra}
\usepackage{xcolor}
\usepackage[colorlinks=true]{hyperref}
\hypersetup{urlcolor=blue, citecolor=red, linkcolor=blue}
\usepackage[toc,page]{appendix}
\numberwithin{equation}{section}
\theoremstyle{plain}
\newtheorem{theorem}{Theorem}[section]

\newtheorem{lemma}[theorem]{Lemma}

\newtheorem{corollary}[theorem]{Corollary}

\theoremstyle{definition}
\newtheorem{definition}[theorem]{Definition}

\newtheorem{remark}[theorem]{Remark}

\newcommand{\mr}{\mathbb{R}}
\newcommand{\ud}{\mathrm{d}}
\allowdisplaybreaks

\begin{document}
	\title{\Large \bf The total Q-curvature, volume entropy and polynomial growth polyharmonic functions}
	\author{Mingxiang Li\thanks{M. Li, Nanjing University,  Email: limx@smail.nju.edu.cn. }}
	\date{}
	\maketitle
\begin{abstract}
	In this paper, we investigate a conformally flat and complete manifold $(M,g)=(\mathbb{R}^n,e^{2u}|dx|^2)$ with finite total Q-curvature. We introduce a new volume entropy, incorporating the background Euclidean metric, and demonstrate that the metric $g$ is normal if and only if the volume entropy is finite. Furthermore, we establish an identity for the volume entropy utilizing the integrated Q-curvature. Additionally, under normal metric assumption,  we get a result concering the behavior of the geometric distance at infinity compared with Euclidean distance. With help of this result,  we prove that each polynomial growth polyharmonic function on such manifolds is of finite dimension. Meanwhile, we prove several rigidity results by imposing restrictions on the sign of the Q-curvature. Specifically,  we establish that on such manifolds, the Cohn-Vossen inequality achieves equality if and only if each polynomial growth polyharmonic function is a constant.
\end{abstract}	

{\bf Keywords: } Total Q-curvature, Normal solution, Volume entropy, Polyharmonic functions.
\medskip

{\bf MSC2020: } 53C18, 53C20, 58J90.
	
\section{Introduction}

When considering a complete surface $(M^2,g)$, the famous Chern-Gauss-Bonnet formula was extended to non-compact surfaces by Cohn-Vossen \cite{CV} and Huber \cite{Hu}. Specifically, if the Gaussian curvature $K$ is absolutely integrable, then the formula states that:
\begin{equation}\label{huber's inequailty}
	\int_MK\ud v_g\leq 2\pi\chi(M)
\end{equation}
where $\chi(M)$ is the Euler number of $M$.  Notably, in \cite{Hu}, Huber demonstrated that such a surface $M$ can be conformally transformed into a closed surface with a finite number of punctures. To explore the generalized version of Huber's theorem, we refer the interested reader to \cite{Chang-Q-Yang2}, \cite{CH}, \cite{MQ} and the relevant references therein. Additionally, Finn \cite{Fin} established that the deficit is a sum of non-negative isoperimetric ratios $\nu_i$ near infinity, expressed as follows:
\begin{equation}\label{Fin's identity}
	\chi(M)-\frac{1}{2\pi}\int_M K\ud v_g=\sum \nu_i.
\end{equation}
In higher-dimensional cases, the Q-curvature performs a role similar to the Gaussian curvature. Naturally, one may wonder if analogous results exist for higher dimensions. For this situation. Chang, Qing, and Yang (\cite{CQY} ,\cite{Chang-Q-Yang2}) achieved a significant breakthrough.
Before presenting their results, we provide some background information for reader's convenience. 

During the 1980s, Paneitz \cite{Paneitz} introduced a renowned conformal operator denoted by
$$P_g^4=\Delta_g^2-\mathrm{div}_g\left((\frac{2}{3}R_gg-2Ric_g)d\right)$$
where $\Delta_g$ represents the Laplace-Beltrami operator, and $R_g$ together with  $Ric_g$ denote the scalar curvature and Ricci curvature tensor of $(M^4,g)$, respectively.  Subsequently, Branson\cite{Bra} introduced the Q-curvature (up to scaling) defined by
$$Q_g^4=-\frac{1}{6}(\Delta_gR_g-R_g^2+3|Ric_g|^2).$$
Remarkably, the Q-curvature satisfies the conformal invariant equation similar to Gaussian curvature:
$$
	P_g^4u+Q_g^4= Q_{\tilde g}^4e^{4u}
$$
for the conformal metric $\tilde g=e^{2u}g$. 
For higher order cases, Graham, Jenne, Mason and Sparling \cite{GJMS} introduced an operator known as GJMS operator  $P_g^{n}$ for $(M^n,g)$ satisfying
$$
	P_g^nu+Q^n_g=Q^n_{\tilde g}e^{nu}
$$
where $\tilde g=e^{2u}g$ and  $n\geq 4$ is an even integer. 
Numerous outstanding  results have been obtained for compact manifolds and for further details, we refer to \cite{CY95}, \cite{Bre}, \cite{DM} and the references therein.

In this paper, we focus on conformally flat cases, i.e.  those of the form $(M,g)=(\mr^n,e^{2u}|dx|^2)$, where $n\geq 2$ is an even integer and the Q-curvature $Q_g$ satisfying the equation:
\begin{equation}\label{Q-curvature }
	(-\Delta)^{\frac{n}{2}}u=Q_ge^{nu}.
\end{equation}
For the case $n=2$, the Q-curvature $Q_g$ is exactly the Gaussian curvature. Meanwhile, when $n\geq 4$, we have the scalar curvature $R_g$ satisfying 
\begin{equation}\label{scalar curvature}
	R_g=2(n-1)e^{-2u}(-\Delta u-\frac{n-2}{2}|\nabla u|^2).
\end{equation} This  equation \eqref{Q-curvature }  holds significant importance across various fields.
Of particular note, when $Q(x)=1$ and the finite volume assumption i.e. $e^{nu}\in L^1(\mr^n)$  holds, classification theorems outlined in works such as \cite{CL91}, \cite{Lin}, \cite{WX}, \cite{Xu JFA} and \cite{Mar MZ} are highly instrumental in tackling many related problems. We define a conformally flat manifold $(\mr^n,e^{2u}|dx|^2)$ to have finite total Q-curvature if
\begin{equation}\label{finite total Q-curvautre}
	\int_{\mr^n}|Q_g|e^{nu}\ud x<+\infty.
\end{equation}
For simplicity, we set a normalized integrated Q-curvature as
$$\alpha_0:=\frac{2}{(n-1)!|\mathbb{S}^n|}\int_{\mr^n}Q_ge^{nu}\ud x$$
throughout this paper where  $|\mathbb{S}^n|$ denotes  the volume of standard sphere in $\mathbb{R}^{n+1}$.

The assumption of finite total Q-curvature is particularly useful as it allows us to represent solutions to \eqref{Q-curvature } as logarithmic potentials known as normal solutions (see Definition \ref{def: normal solution}) in certain cases. Normal solutions play a crucial role in various works, including \cite{Fin}, \cite{Hu}, \cite{CQY}, \cite{Chang-Q-Yang2}, and many others. Furthermore, if $u$ is a normal solution, the conformal metric $g=e^{2u}|dx|^2$ is considered as a normal metric. Throughout this paper, we just focus on the smooth metric $g=e^{2u}|dx|^2$. For general complete locally conformally flat manifolds with simple ends, we will discuss in our forthcoming paper.
 
For the sake of convenience, we use the notation $B_R(p)$ to refer to an Euclidean  ball in $\mr^n$ centered at $p\in\mr^n$ with a radius of $R$. Moreover, we use $V_g(\cdot)$ to denote the measure with respect to the metric $g=e^{2u}|dx|^2$, and $d_g(\cdot,\cdot)$ represents the geodesic distance. Also, $|B_R(p)|$ refers to the volume of $B_R(p)$ in terms of the Euclidean metric.  For a function $\varphi(x)$, the positive part of $\varphi(x)$ is denoted as $\varphi(x)^+$ 
and the negative part of $\varphi(x)$ is denoted as $\varphi(x)^-$.  Set
$\fint_{E}\varphi(x)\ud x=\frac{1}{|E|}\int_{E}\varphi(x)\ud x$
for any measurable set $E$.  For a constant $C$, $C^+$ denotes $C$ if $C\geq 0$, otherwise, $C^+=0$. 
 Here and thereafter, we denote by $C$ a constant which may be different from line to line. For $s\in \mr$, $[s]$ denotes the largest integer not greater than $s$.

In their work \cite{CQY}, Chang, Qing and Yang derived a Chern-Gauss-Bonnet formula for a complete manifold $(\mr^4,e^{2u}|dx|^2)$ generalizing  Cohn-Vossen inequality \eqref{huber's inequailty} and Fin's identity \eqref{Fin's identity} to higher dimensional manifolds with help of Q-curvature. 	Supposing  $g=e^{2u}|dx|^2$ is a  complete metric on $\mr^4$ with finite total Q-curvature and $R_g\geq 0$ near infinity, then there holds
\begin{equation}\label{CQY indentity}
	1-\frac{1}{8\pi^2}\int_{\mr^4}Q_ge^{4u}\ud x=\lim_{R\to\infty}\frac{V_g(\partial B_R(0))^{\frac{4}{3}}}{4(2\pi^2)^{\frac{1}{3}}V_g(B_R(0))}.
\end{equation}
In fact, the results hold for higher order cases as well, and references such as \cite{Chang-Q-Yang2}, \cite{Fa}, and \cite{NX} can provide more information. It's worth noting that the right side of the aforementioned identity represents an isoperimetric ratio.   Bonk, et al. \cite{BHS} and Wang (\cite{Wang IMRN}, \cite{Wang}) have studied the isoperimetric inequality on such manifolds, under additional assumptions on the value of the integrated Q-curvature, using the properties of the strong $A_\infty$ weight, which will be discussed in Section \ref{3}. Similar results can also be found in \cite{LT}, \cite{ACT}, \cite{Xiao} and the references therein.

In several notable works, including \cite{Mann}, \cite{LW}, and the references therein, the concept of "volume entropy" has been utilized to characterize various properties of manifolds. In this paper, we introduce a new volume entropy $\tau(g)$ associated with the conformal background metric, which differs from the definition given in \cite{Mann}. The volume entropy $\tau(g)$ of the metric $g=e^{2u}|dx|^2$ is defined as follows:
\begin{equation}\label{def: volume entropy}
	\tau(g):=\lim_{R\to\infty}\sup\frac{\log V_g(B_R(0))}{\log|B_R(0)|}.
\end{equation}
We say the volume entropy of the metric $g$ is finite if $\tau(g)<+\infty.$
As mentioned earlier, the  normal metric plays a crucial role in dealing with conformally flat manifolds. Unlike the approach in \cite{CQY}, which relies on restrictions regarding the sign of scalar curvature near infinity, our first result employs the concept of volume entropy. By utilizing volume entropy, we establish  a necessary and sufficient condition for a complete metric with finite total Q-curvature to be classified as a normal metric.
\begin{theorem}\label{thm: finite volume entropy and normal metric}
	Consider a conformally flat and  complete manifold $(M,g)=(\mr^n,e^{2u}|dx|^2)$  with finite total Q-curvature  where   $n\geq 2$ is an even integer. 
	\begin{enumerate}[(i)]
		\item The metric  $g$ is normal if and only if the volume entropy $\tau(g)$ is finite. 
		\item If the volume entropy $\tau(g)$ is finite, there holds
		\begin{equation}\label{tau (g) identity}
			\tau(g)=1-\alpha_0.
		\end{equation}
	\end{enumerate}

\end{theorem}

\begin{remark}
In \cite{Hu}, for the two-dimensional case, the finite total Gaussian curvature assumption and the  completeness of the metric are already sufficient to ensure the normal metric. However, for higher dimensional cases, the finite total Q-curvature is not enough to ensure that (refer to Remark 1.2 in \cite{CQY}), mainly because the complexity of the kernel of $P_g$. Besides, the completeness  of  metric is also crucial since there exist non-normal solutions with finite total Q-curvature and finite volume entropy(See \cite{Chang-Chen}, \cite{Hyder}, \cite{WY}).  In \cite{Zhang}, Zhang showed that the metric is normal under the assumption  on  the decay rate of Q-curvature near infinity.
\end{remark}
\begin{corollary}\label{cor}
	Consider a conformally flat and  complete manifold $(M,g)=(\mr^n,e^{2u}|dx|^2)$  with finite total Q-curvature  where   $n\geq 2$ is an even integer.  If the volume entropy $\tau(g)$ is finite and $Q_g\geq 0$, then
	$$\tau(g)\leq 1$$
	with equality holds if and only if $u\equiv C.$
\end{corollary}
In Li and Tam's  work \cite{LT}, they focus on a key geometric result concerning the behavior of geodesic distance at infinity when compared to a flat metric background. Their main objective is to prove Yau's conjecture \cite{Yau} regarding the dimensions of polynomial growth of harmonic functions. For more detailed information, interested readers can refer to \cite{Li-Yau}, \cite{CM Ann}, \cite{CM}, and \cite{CCM}.  For higher dimensional cases, we also obtain a distance comparison identity under the normal metric assumption.
\begin{theorem}\label{thm: length comparison identity}
	Consider a conformally flat manifold $(M,g)=(\mr^n,e^{2u}|dx|^2)$  with finite total Q-curvature  where   $n\geq 2$ is an even integer.  Supposing that the metric $g$ is normal, then for each fixed point $p$, there holds
$$ \lim_{|x|\to\infty}\frac{\log d_g(x,p)}{\log|x-p|}=\left(1-\alpha_0\right)^+.$$
Moreover, if 
$\alpha_0>1$,
one has 
$$\mathrm{diam}(\mr^n,e^{2u}|dx|^2)<+\infty$$
where the diameter is defined by  $\mathrm{diam}(M,g):=\sup_{x,y\in M} d_g(x,y).$
\end{theorem}
In \cite{CQY}, Chang, Qing and Yang demonstrated that if the metric is normal and complete, then $\alpha_0 \leq 1$. It is interesting to consider whether we can provide a restriction on $\alpha_0$ and deduce the completeness of the metric. In the case of two dimensions, Aviles provided a result (see Theorem C in \cite{Av}). Here, we can utilize Theorem \ref{thm: length comparison identity} to demonstrate that the metric is complete if $\alpha_0 < 1$.
\begin{corollary}\label{cor: reversed complete}
	Consider a conformally flat manifold $(M,g)=(\mr^n,e^{2u}|dx|^2)$  with finite total Q-curvature  where   $n\geq 2$ is an even integer.  Suppose that the metric $g$ is normal. If  $\alpha_0<1$, then the metric $g$ is complete.
\end{corollary}
\begin{remark}
	If $\alpha_0 = 1$, the situation becomes more subtle. The metric can be either complete or non-complete. In Section \ref{section: example}, we will construct some examples to confirm this. 
\end{remark}

In the context of higher-dimensional manifolds $(\mathbb{R}^n, e^{2u}|dx|^2)$, it is natural to consider the kernel of the GJMS operator $P_g$ on such manifolds and we refer to these kernel functions as polyharmonic functions.  Precisely, for  $d\geq0$,  we define $\mathcal{PH}_d(M,g)$ as the linear space of the kernel function of $P_g$  of polynomial growth at most $d$: for some fixed point $p\in M$
$$\mathcal{PH}_d(M,g):=\{f(x)|P_g f(x)=0, |f(x)|\leq C(d_g(x,p)^d+1)\}.$$
Similarly, define the polynomial growth polyharmonic functions on $\mr^n$ respect to standard Euclidean metric as follows:
$$\mathcal{PH}_d(\mr^n,|dx|^2):=\{f(x)|(-\Delta)^{\frac{n}{2}}f=0,|f(x)|\leq C(|x|^d+1)\}.$$ 
With help of a Liouville-type theorem, it is well-known that such linear  spaces consists of polynomials  (Seeing   \cite{ABR},\cite{ACL}) and the dimension of such linear space is finite for each $d\geq 0$. 

Taking inspiration from \cite{LT} in the case of two dimensional surfaces, our objective is to establish a framework for analyzing the linear spaces $\mathcal{PH}_d(M,g)$ of polynomial growth polyharmonic functions on conformally flat manifolds. It is important to note that the methodology employed in \cite{LT} may not be directly applicable to higher dimensional scenarios. Instead, our current approach builds upon the foundation laid in our previous work in \cite{Li23}. Similar to Theorem 4.6 in \cite{LT}, if $Q_g$ is non-negative outside a compact set, we will obtain an identity. 
Moreover, if $Q_g\geq 0$, we are able to derive a rigidity result similar to the findings presented in \cite{CCM} for manifolds with non-negative Ricci curvature.
 Through the utilization of Fin's identity  \eqref{Fin's identity} and \eqref{CQY indentity},  we know that the isoperimetric ratio vanishes when the Cohn-Voseen inequality \eqref{huber's inequailty} achieves the identity. Here, we would like to  provide an alternative characterization using polynomial growth polyharmonic functions.
\begin{theorem}\label{thm: PH dimension is finite }
Consider a conformally flat and complete manifold $(M,g)=(\mathbb{R}^n,e^{2u}|dx|^2)$, where $n\geq 2$ is an even integer, and the total Q-curvature is finite. Assume that the volume entropy $\tau(g)$ is finite. Then the  following results are established:
\begin{enumerate}[(i)]
	\item For each $d\geq 0$,  there holds 
	$$\dim\left(\mathcal{PH}_d(M,g)\right)\leq\dim(\mathcal{PH}_{d\tau(g)}(\mr^n,|dx|^2)).$$
	\item If $Q_g$ is non-negative outside a compact set and $d\geq0$, there holds
	\begin{equation}\label{identiy of the dimension}
		\dim(\mathcal{PH}_d(M,g))=\dim(\mathcal{PH}_{d\tau(g)}(\mr^n,|dx|^2)).
	\end{equation}
	\item If $Q_g\geq 0$ everywhere,  for each integer $k\geq 1$,
	$$\dim(\mathcal{PH}_k(M,g))\leq \dim(\mathcal{PH}_{k}(\mr^n,|dx|^2))$$
	with equality holds if and only if $u\equiv C.$
	
	\item The volume entropy  $\tau(g)=0$  if and only if for each $d\geq 0$,
	$\mathcal{PH}_d(M,g)=\{\mathrm{constant \;functions}\}.$
	
\end{enumerate}

\end{theorem}

Here is a brief overview of the structure of this paper. In Section \ref{2}, we establish some properties of the logarithmic potential and give a decomposition result in Theorem \ref{thm:decomposition} under the assumption of finite volume entropy. We also provide an estimate for the volume entropy in Theorem \ref{thm: lambda leq 1-alpha}. As a practical application, we get a lower bound of the integrated Gaussian curvature by considering a finite volume assumption instead of the completeness of the metric in Theorem \ref{thm: reversed Cohn-Voseen} and Theorem \ref{thm: reversed Cohn-Voseen inequality for Q-curvature}.  In Section \ref{3}, we discuss some properties of strong $A_\infty$ weights and give the proof of Theorem \ref{thm: length comparison identity}.  Meanwhile, we derive a volume and length comparison identity in Theorem \ref{thm:d_g(x,p) and |x-p|} if the metric is normal and complete.  In Section \ref{section: iff normal solution}, we introduce a criterion for the solution \eqref{Q-curvature } to be normal solutions in Theorem \ref{thm: necessary and sufficient for normal solution} related to scalar curvature.  Meanwhile, we give the proofs of Theorem \ref{thm: finite volume entropy and normal metric} and Corollary \ref{cor}.    In Section \ref{section: pg pf}, we prove Theorem \ref{thm: PH dimension is finite }    with help of the geodesic distance comparion ideneity in Theorem \ref{thm:d_g(x,p) and |x-p|}. Finally, in Section \ref{section: example}, we provide some examples to help illustrate the subtle case when $\alpha_0=1$.

\section{Asymptotic behavior of logrithmic potential}\label{2}
To enhance our understanding of the logarithmic potential on non-compact manifolds, it is helpful to first discuss Green's representation on compact manifolds. 
 On a four-dimensional compact manifold $(M,g_0)$, if the kernel of the Paneitz operator $P_{g_0}$ consists of constants, the Green's function $G_{g_0}(x,y)$ exists (see Lemma 1.7 in \cite{CY95}), and we have a nice representation of the solution to the equation
$P_{g_0} u=\varphi(x)$
as 
$$u(x)-\frac{1}{V_{g_0}(M)}\int_Mu\ud\mu_{g_0} =\int_MG_{g_0}(x,y)\varphi(y)\ud \mu_{g_0}.$$
Equivalently, we can rewrite it as
\begin{equation}\label{compact Green function}
	u(x)=\int_M(G_{g_0}(x,y)-G_{g_0}(x_0,y))\varphi(y)\ud\mu_{g_0} +u(x_0).
\end{equation}
Returning to our current setting, we note that on $\mr^n$, the kernel of the polyharmonic Laplacian operator $(-\Delta)^{\frac{n}{2}}$ contains more than just constants, which brings significant difficulties to our research.
Consider the following equation for even integer $n\geq 2$ given by the equation
\begin{equation}\label{w equation}
	(-\Delta )^{\frac{n}{2}}w(x)=f(x),
\end{equation}
where $w(x)$ and $f(x)$ are both smooth functions on $\mathbb{R}^n$ with $f\in L^1(\mr^n)$. We aim to investigate whether the solutions to \eqref{w equation} can be represented analogously to \eqref{compact Green function} under certain assumptions. 
Actually, the Green's function of the polyharmonic Laplacian  operator $(-\Delta)^{\frac{n}{2}}$
satisfying (See Chapter 1.2 in \cite{ACL})
\begin{equation}\label{Green's function}
	(-\Delta)^{\frac{n}{2}}\left(\frac{2}{(n-1)!|\mathbb{S}^n|}\log\frac{1}{|x|}\right)=\delta_0(x)
\end{equation}
where  $\delta_0(x)$ is the  Dirac operator.
Compared with \eqref{compact Green function}, we naturally define a logarithmic potential as follows:
$$\mathcal{L}(f)(x)=\frac{2}{(n-1)!|\mathbb{S}^n|}\int_{\mr^n}\log\frac{|y|}{|x-y|}f(y)\ud y.$$
Since smooth $f\in L^1(\mr^n)$, it is not hard to check that the above  logrithmic potential is well-defined.
Our definition of the logarithmic potential differs slightly from a similar definition introduced in \cite{BHS}. We should point out that the logarithmic potential has been previously introduced and studied in various forms and special cases by many researchers, including in \cite{CL91}, \cite{Lin}, and other works.

\begin{definition}\label{def: normal solution}
	We say $w(x)$ is a normal solution to \eqref{w equation} if for some constant $C$,
	$$w(x)=\mathcal{L}(f)(x)+C.$$
\end{definition}

We introduce a quantity like the volume entropy $\tau(g)$ to measure the volume growth of $w(x)$ on $\mr^n$, denoted as $\mathcal{V}_{sup}(e^{nw})$ and defined by the formula:

$$\mathcal{V}_{sup}(e^{nw}):=\lim_{R\to\infty}\sup\frac{\log\int_{B_R(0)}e^{nw}\ud x}{\log |B_R(0)|}.$$
Similarly, we define $\mathcal{V}_{inf}(e^{nw})$ as:
$$\mathcal{V}_{inf}(e^{nw}):=\lim_{R\to\infty}\inf\frac{\log\int_{B_R(0)}e^{nw}\ud x}{\log |B_R(0)|}.$$
If $\mathcal{V}_{sup}(e^{nw})<+\infty$ with  $\mathcal{V}_{sup}(e^{nw})=\mathcal{V}_{inf}(e^{nw}),$ we say that the limit of volume growth exists and denote it as $\mathcal{V}(e^{nw})$.
The  function $w(x)$ is said to have polynomial volume growth if $\mathcal{V}_{sup}(e^{nw})$ is finite.  In fact, one may verify that $\mathcal{V}_{sup}(e^{nw})$ is finite if and only if there exists $s\geq 0$ such that 
$$\int_{B_R(0)}e^{nw}\ud x= O(R^s).$$

Assuming polynomial volume growth of $w(x)$,  we can decompose the solution $w(x)$ of equation \eqref{w equation} into the logarithmic potential and a polynomial. Furthermore, we can derive a necessary and sufficient condition for being normal solutions for \eqref{w equation}, which is of particular interest.
\begin{theorem}\label{thm:decomposition}
	Let $n\geq 2$ be an even integer and consider a smooth solution $w(x)$ to \eqref{w equation} with smooth $f\in L^1(\mathbb{R}^n)$. Suppose that $w(x)$  has  polynomial volume growth.
	Then we have a decomposition
	\begin{equation}\label{w=Lw+p}
		w(x)=\mathcal{L}(f)(x)+P(x)
	\end{equation}
	where $P(x)$ is a polynomial with degree at most $n-2$ and  $P(x)\leq C$ for some constant $C$.  
	
	Moreover,  for $n\geq 4$,  $P(x)\equiv C$ if and only if one of the following conditions holds
	\begin{enumerate}[(a)]
		\item 	\begin{equation}\label{Delta w=o(R^n)}
			\int_{B_R(0)}|\Delta w|\ud x=o(R^{n}).
		\end{equation}
		\item	\begin{equation}\label{w=o(R^n+2)}
			\int_{B_R(0)}| w|\ud x=o(R^{n+2}).
		\end{equation}
	\end{enumerate}
\end{theorem}

Before establising the proof of Theorem \ref{thm:decomposition}, we need introduce some helpful lemmas. For brevity, set the notation $\alpha$ defined as 
$$\alpha=\frac{2}{(n-1)!|\mathbb{S}^n|}\int_{\mr^n}f(x)\ud x.$$ It is natural to expect that the logarithmic potential $\mathcal{L}(f)(x)$ has an asymptotic behavior near infinity resembling $-\alpha\log|x|$. In previous works such as \cite{CLin} and \cite{CL93}, such behavior has been described under additional decay assumptions to obtain pointwise estimates. In the current paper, we   will obtain integral results without requirng any additional assumptions.
\begin{lemma} \label{lem: L(f)}
Given  $f(x)$ satisfying  $f\in L^\infty_{loc}(\mr^n)$ and  $f\in L^1(\mr^n)$ with even integer $n\geq 2$, for $|x|\gg1$, there holds
	\begin{equation}\label{Lf=-aplha log x+}
\mathcal{L}(f)(x)=(-\alpha+o(1))\log|x|+\frac{2}{(n-1)!|\mathbb{S}^n|}\int_{B_1(x)}\log\frac{1}{|x-y|}f(y)\ud y
\end{equation}
where  $o(1)\to 0$ as $|x|\to\infty$.
\end{lemma}
\begin{proof}
	Based on the assumptions that $f\in L^\infty_{loc}(\mathbb{R}^n)$ and $f\in L^1(\mathbb{R}^n)$, it is not difficult to verify that $\mathcal{L}(f)$ is well-defined.
	Choose $|x|\geq e^4$ such  that $|x|\geq 2\log|x|$. Following the argument in \cite{Li23}, we split $\mr^n$ into three pieces
	$$A_1=B_{1}(x), \quad A_2=B_{\log|x|}(0),\quad A_3=\mr^n\backslash (A_1\cup A_2).$$
	For $y\in A_2$ and $|y|\geq2$ , we have $|\log\frac{|x|\cdot|y|}{|x-y|}|\leq \log(2\log|x|)$.
	Respectively,  for $|y|\leq 2$, $ |\log\frac{|x|\cdot|y|}{|x-y|}|\leq |\log|y||+C$.
	Thus 
	\begin{equation}\label{A_2}
		|\int_{A_2}\log\frac{|y|}{|x-y|}f\ud y+\log|x|\int_{A_2}f\ud y|\leq C\log\log|x|+C=o(1)\log|x|.
	\end{equation}
	as $|x|\to \infty$.
	For $y\in A_3$, it is not hard to check 
	$$\frac{1}{|x|+1}\leq\frac{|y|}{|x-y|}\leq |x|+1.$$ With help  of this estimate, we could  control the integral over $A_3$ as
	\begin{equation}\label{A_3}
		|\int_{A_3}\log\frac{|y|}{|x-y|}f\ud y|\leq \log(|x|+1)\int_{A_3}|f|\ud y.
	\end{equation}
	For $y\in B_1(x)$, one has
	$1\leq |y|\leq |x|+1$ and then
	$$|\int_{A_1}\log|y|f\ud y|\leq \log(|x|+1)\int_{A_1}|f|\ud y.$$
		Since $f\in L^1(\mr^n)$, notice that $\int_{A_3\cup A_1}|f|\ud y\to 0$ as $|x|\to \infty$ and 
	$$\frac{2}{(n-1)!|\mathbb{S}^n|}\int_{A_2}f(y)\ud y=\alpha+o(1).$$
  Thus there holds
	\begin{equation}\label{u =-aplha log x1}
		\mathcal{L}(f)(x)=(-\alpha+o(1))\log|x|+\frac{2}{(n-1)!|\mathbb{S}^n|}\int_{B_1(x)}\log\frac{1}{|x-y|}f(y)\ud y.
	\end{equation}
\end{proof}

We will now utilize the lemma above to establish integral results pertaining to $\mathcal{L}(f)$ over different balls. These estimates will play a crucial role in our proofs for the main theorems. 
 We would like to clarify beforehand: the following lemmas  are under the same assumptions as satated in Lemma \ref{lem: L(f)}.
\begin{lemma}\label{lem: B_1 L(f)}
	For any $r_0>0$ fixed, there holds
	\begin{equation}\label{B_1(x)Lf}
		\fint_{B_{r_0}(x)}\mathcal{L}(f)(y)\ud y=(-\alpha+o(1))\log|x|.
	\end{equation}
\end{lemma}
\begin{proof}
	Fubini's theorem yields that
	\begin{align*}
	&	|\int_{B_{r_0}(x)}\int_{B_1(z)}\log\frac{1}{|z-y|}f(y)\ud y\ud z|\\
\leq &	\int_{B_{r_0}(x)}\int_{B_{r_0+1}(x)}|\log\frac{1}{|z-y|}|\cdot|f(y)|\ud y\ud z\\
		\leq & \int_{B_{r_0+1}(x)}|f(y)|\ud y\int_{B_{2r_0+1}(0)}|\log|z||\ud z\\
		\leq & C.
	\end{align*}
	For $|x|\gg1$, there holds
	for any $y\in B_{r_0}(x)$,
	$$ |\log\frac{|y|}{|x|}|\leq C.$$
	With help of Lemma \ref{lem: L(f)}, then we have
	$$\fint_{B_{r_0}(x)}\mathcal{L}(f)(y)\ud y=(-\alpha+o(1))\log|x|.$$ 
\end{proof}

\begin{lemma}\label{lem:B_r_0|x|}
	For any  $0<r_1<1$ fixed, there holds
		\begin{equation}\label{B_|x|/2v}
		\fint_{B_{r_1|x|}(x)}\mathcal{L}(f)(y)\ud y=(-\alpha +o(1))\log|x|.
	\end{equation}
\end{lemma}
\begin{proof}
	 By direct computation and Fubini's theorem, one has 
	\begin{align*}
		&\int_{B_{r_1|x|}(x)}|\int_{B_1(z)}\log\frac{1}{|z-y|}f(y)\ud y|\ud z\\
		\leq & 	\int_{B_{r_1|x|}(x)}\int_{B_1(z)}\frac{1}{|z-y|}|f(y)|\ud y\ud z\\
		\leq &\int_{B_{r_1|x|}(x)}\int_{B_{r_1|x|+1}(x)}\frac{1}{|z-y|}|f(y)|\ud y\ud z\\
		\leq &\int_{B_{r_1|x|+1}(x)}|f(y)|\int_{B_{r_1|x|}(x)}\frac{1}{|z-y|}\ud z\ud y\\
		\leq & \int_{B_{r_1|x|+1}(x)}|f(y)|\int_{B_{2r_1|x|+1}(0)}\frac{1}{|z|}\ud z\ud y\\
		\leq &C|x|^{n-1}.
	\end{align*}
	Thus 
	\begin{equation}\label{B_r_0|x|Q}
		\fint_{B_{r_1|x|}(x)}\int_{B_1(z)}\log\frac{1}{|z-y|}f(y)\ud y\ud z=O(|x|^{-1}).
	\end{equation}
	Meanwhile, for  $y\in B_{r_1|x|}(x)$, there holds
	\begin{equation}\label{log|x|/|x_0|}
	|\log\frac{|y|}{|x|}|\leq \log(1-r_1)+\log(1+r_1)\leq C.
	\end{equation}
	With help of these estimates \eqref{B_r_0|x|Q}, \eqref{log|x|/|x_0|} and 
	Lemma \ref{lem: L(f)},  we have
	\begin{equation}\label{int B_|x|/2 v}
		\fint_{B_{r_1|x|}(x)}\mathcal{L}(f)(y)\ud y=(-\alpha+o(1))\log|x|.
	\end{equation}
\end{proof}

\begin{lemma}\label{lem:B_x^-r_1}
	For any  $r_2>0$ fixed, there holds
	\begin{equation}\label{B_|x|/2v}
		\fint_{B_{|x|^{-r_2}}(x)}\mathcal{L}(f)(y)\ud y=(-\alpha +o(1))\log|x|.
	\end{equation}
\end{lemma}
\begin{proof}
	 Firstly, we need mofidy the estimate of Lemma \ref{lem: L(f)}. 
	By direct computation,  for $|x|\gg1$, we have
	$$|\int_{B_1(x)\backslash B_{|x|^{-r_2}}(x)}\log\frac{1}{|x-y|}f(y)\ud y|\leq r_2\log|x|\int_{B_1(x)}|f|\ud y=o(1)\log|x|.$$
	Using Lemma \ref{lem: L(f)}, there holds
	$$\mathcal{L}(f)(x)=(-\alpha +o(1))\log|x|+\frac{2}{(n-1)!|\mathbb{S}^n|}\int_{B_{|x|^{-r_2}}(x)}\log\frac{1}{|x-y|}f(y)\ud y.$$
	When $|x|\geq 2$ and $r_2>0$, if $z\in B_{|x|^{-r_2}}(x)$ and $y\in B_{|z|^{-r_2}}(z)$, one has
	$$|y-x|\leq|y-z|+|z-x|\leq |z|^{-r_2}+|x|^{-r_2}\leq (|x|-|x|^{-r_2})^{-r_2}+|x|^{-r_2}\leq 3|x|^{-r_2}.$$
	By using Fubini's theorem as before, for $|x|\gg1$, there holds
	\begin{align*}
		&|\int_{B_{|x|^{-r_2}}(x)}\int_{B_{|z|^{-r_2}}(z)}\log\frac{1}{|z-y|}f(y)\ud y\ud z|\\
		\leq &\int_{B_{|x|^{-r_2}}(x)}\int_{B_{3|x|^{-r_2}}(x)}\log\frac{1}{|z-y|}|f(y)|\ud y\ud z\\
		\leq &\int_{B_1(x)}|f(y)|\ud y\int_{B_{4|x|^{-r_2}}(0)}\log\frac{1}{|z|}\ud z\\
		=&(\int_{B_1(x)}|f(y)|\ud y)\cdot|\mathbb{S}^n|(4|x|^{-r_2})^n\frac{1-n\log(4|x|^{-r_2})}{n^2}
	\end{align*}
	which yields that
	\begin{align*}
		&|\frac{1}{|B_{|x|^{-r_2}}(x)|}\int_{B_{|x|^{-r_2}}(x)}\int_{B_{|z|^{-r_2}}(z)}\log\frac{1}{|z-y|}f(y)\ud y\ud z|\\
		\leq &C\left(\int_{B_1(x)}|f(y)|\ud y\right)(nr_2\log|x|+C)\\
		=&o(1)\log|x|
	\end{align*}
	since $\int_{B_1(x)}|f(y)|\ud y=o(1)$.
For any $y\in B_{|x|^{-r_2}}(x)$, it is easy to check
$$|\log\frac{|y|}{|x|}|\leq C.$$
 Thus we have
	\begin{equation}\label{int B_|x|^3-n}
		\fint_{B_{|x|^{-r_2}}(x)}\mathcal{L}(f)(z)\ud z=(-\alpha+o(1))\log |x|.
	\end{equation}
\end{proof}

\begin{lemma}\label{lem: e^nL(f)}
	For $r_3>0$ fixed, there holds
	$$\log\left(\fint_{B_{r_3}(x)}e^{n\mathcal{L}(f)(y)}\ud y\right)=(-n\alpha+o(1))\log|x|.$$
\end{lemma}
\begin{proof}
	On one hand, with help of Jensen's inequality and Lemma \ref{lem: B_1 L(f)},  one has
	\begin{equation}\label{e^Lf lower bound}
		\fint_{B_{r_3}(x)}e^{n\mathcal{L}(f)(y)}\ud y\geq \exp\left(\fint_{B_{r_3}(x)}n\mathcal{L}(f)(y)\ud y\right)=e^{(-n\alpha+o(1))\log|x|}.
	\end{equation}
Now, we are going to deal with  the upper bound.
For $|y|\gg1$, there holds
$$|\int_{B_1(y)\backslash B_{1/4}(y)}\log\frac{1}{|y-z|}f(z)\ud z|\leq C\int_{B_1(y)\backslash B_{1/4}(y)}|f(z)|\ud z.$$
Combing with Lemma \ref{lem: L(f)}, we have
$$\mathcal{L}(f)(y)=(-\alpha +o(1))\log|y|+\frac{2}{(n-1)!|\mathbb{S}^n|}\int_{B_{1/4}(y)}\log\frac{1}{|y-z|}f(z)\ud z.$$
Then for $|x|\gg 1$ and $y\in B_{1/4}(x)$, 
 one has
\begin{equation}\label{L (f)leq -alpha log|x|}
	\mathcal{L}(f)(y)\leq (-\alpha +o(1))\log|x|+\frac{2}{(n-1)!|\mathbb{S}^n|}\int_{B_{1/2}(x)}\log\frac{1}{|y-z|}f^+(z)\ud z
\end{equation}
where we have used  the estimate $|y-z|\leq 1$ in this case.

We claim that for $|x|\gg 1$,
\begin{equation}\label{e^nLf leq |x|^n-alpha}
	\int_{B_{1/4}(x)}e^{n\mathcal{L}(f)(y)}\ud y\leq e^{(-n\alpha +o(1))\log|x|}.
\end{equation}
If $f^+(z)=0$ a.e. on $B_{1/2}(x)$, we immediately obatain \eqref{e^nLf leq |x|^n-alpha} due to \eqref{L (f)leq -alpha log|x|}.
Otherwise, Jensen's inequality yields that
\begin{align*}
	&\int_{B_{1/4}(x)}e^{n\mathcal{L}(f)(y)}\ud y\\
	\leq & |x|^{-n\alpha+o(1)}\int_{B_{1/4}(x)}\exp\left(\frac{2n}{(n-1)!|\mathbb{S}^n|}\int_{B_{1/2}(x)}\log\frac{1}{|y-z|}f^+(z)\ud z\right)\ud y\\
	\leq & |x|^{-n\alpha+o(1)}\int_{B_{1/4}(x)}\int_{B_{1/2}(x)}|y-z|^{-\frac{2n\|f^+\|_{L^1(B_{1/2}(x))}}{(n-1)!|\mathbb{S}^n|}}\frac{f^+(z)}{\|f^+\|_{L^1(B_{1/2}(x))}}\ud z\ud y.
\end{align*}
Since $f\in L^1(\mr^n)$, there exists $R_2>0$ such that $|x|\geq R_2$, we have $$\|f^+\|_{L^1(B_{1/2}(x))}\leq \frac{(n-1)!|\mathbb{S}^n|}{4n}.$$
Applying Fuibini's theorem,  we prove
the claim.
For $r_3>0 $ fixed ,  we can choose finite balls $1\leq j\leq C(r_3)$ such that $B_{r_3}(x)\subset \cup_j B_{1/4}(x_j)$ with $x_j\in B_{r_3}(x)$.
Hence, using the estimate \eqref{e^nLf leq |x|^n-alpha}, for $|x|\gg 1$, we have
$$\log\fint_{B_{r_3}(x)}e^{n\mathcal{L}(f)(y)}\ud y\leq (-n\alpha+o(1))\log|x|.$$
Combing with  \eqref{e^Lf lower bound}, we finish our proof.
\end{proof}
\begin{lemma}\label{lem: e^nLf on B_R+1 B_R-1}
	For $R\gg1$, there holds
	$$\fint_{B_{R+1}(0)\backslash B_{R-1}(0)}e^{n\mathcal{L}(f)}\ud y=R^{-n\alpha+o(1)}$$
	where $o(1)\to 0$ as $R\to\infty.$
\end{lemma}
\begin{proof}
	By a direct computation, we obtain the inequality $C^{-1}R^{n-1} \leq |B_{R+1}(0)\backslash B_{R-1}(0)| \leq C R^{n-1}$ for $R \gg 1$, where $C$ is independent of $R$. We can select an index $C^{-1}R^{n-1} \leq i_R \leq C R^{n-1}$ such that the balls $B_{1/4}(x_j)$ with $|x_j| = R$ and $1 \leq j \leq i_R$ are pairwise disjoint, and the sum of the balls $B_4(x_j)$ cover the annulus $B_{R+1}(0)\backslash B_{R-1}(0)$. Applying Lemma \ref{lem: e^nL(f)}, we obtain the following result
	\begin{align*}
		\fint_{B_{R+1}(0)\backslash B_{R-1}(0)}e^{n\mathcal{L}(f)}\ud y\leq & \frac{1}{C^{-1} R^{n-1}}\sum_{j=1}^{i_R}\int_{B_4(x_j)}e^{n\mathcal{L}(f)}\ud y\\
		\leq & CR^{1-n}\sum_{j=1}^{i_R}|x_j|^{-n\alpha+o(1)}\\
		\leq &CR^{1-n}\cdot CR^{n-1}\cdot R^{-n\alpha +o(1)}\\
		=&R^{-n\alpha +o(1)}.
	\end{align*}
Similarly,
$$ \fint_{B_{R+1}(0)\backslash B_{R-1}(0)}e^{n\mathcal{L}(f)}\ud y \geq \frac{1}{C R^{n-1}}\sum_{j=1}^{i_R}\int_{B_{1/4}(x_j)}e^{n\mathcal{L}(f)}\ud y= R^{-n\alpha+o(1)}.$$

Finally, we get the desired result.
\end{proof}

\begin{lemma}\label{lem: e^nLf B_R}
	There holds 
	$$\mathcal{V}(e^{n\mathcal{L}(f)})=(1-\alpha)^+.$$ Moreover, if $\alpha>1$, one has
	$$\int_{\mr^n}e^{n\mathcal{L}(f)}\ud x<+\infty.$$
\end{lemma}
\begin{proof}
	For $R\gg1$, choose $x_R\in \mr^n$ with  $|x_R|=\frac{R}{2}$. 
	With help of Lemma \ref{lem:B_r_0|x|} and Jensen's equality, we have
	\begin{align*}
		\int_{B_R(0)}e^{n\mathcal{L}(f)}\ud x\geq &\int_{B_{|x_R|/2}(x_R)}e^{n\mathcal{L}(f)}\ud x\\
		\geq &|B_{|x_R|/2}(x_R)|\exp\left(\fint_{B_{|x_R|/2}(x_R)}n\mathcal{L}(f)\ud x\right)\\
		\geq &CR^{n-n\alpha+o(1)}.
	\end{align*}
	By taking $R\to\infty$, we have
	$$\lim_{R\to\infty}\inf \frac{\log\int_{B_R(0)}e^{n\mathcal{L}(f)}\ud x}{\log|B_R(0)|}\geq 1-\alpha.$$
	Since for any $R\geq 1$, 
	$$0<\int_{B_1(0)}e^{n\mathcal{L}(f)}\ud x\leq \int_{B_R(0)}e^{n\mathcal{L}(f)}\ud x,$$
	we obtain that 
	\begin{equation}\label{lambda inf}
		\mathcal{V}_{inf}(e^{n\mathcal{L}(f)})\geq (1-\alpha)^+.
	\end{equation}
	
	Now, we claim that if $\alpha>1$, there holds
	$$\int_{\mr^n}e^{n\mathcal{L}(f)}\ud x<+\infty.$$ 
	 	If $\alpha>1$, due to Lemma \ref{lem: e^nLf on B_R+1 B_R-1}, there exist $R_1>1$ such that for any $R\geq R_1$, one has 
	$$\fint_{B_{R+1}(0)\backslash B_{R-1}(0)}e^{n\mathcal{L}(f)}\ud x\leq R^{-n\alpha +\frac{n(\alpha-1)}{2}}.$$
	Then
	\begin{align*}
		\int_{B_R(0)}e^{n\mathcal{L}(f)}\ud x\leq &	\int_{B_{R_1}(0)}e^{n\mathcal{L}(f)}\ud x+\sum_{i=[R_1]+1}^{[R]}\int_{B_{i+1}(0)\backslash B_{i-1}(0)}e^{n\mathcal{L}(f)}\ud x.\\
		\leq &C+\sum_{i=[R_1]+1}^{[R]}Ci^{n-1}\cdot i^{-n\alpha+\frac{n(\alpha-1)}{2}}\\
		=&C+C\sum_{i=[R_1]+1}^{[R]}i^{-1-\frac{n(\alpha-1)}{2}}.
	\end{align*}
Due to $\alpha>1$, the right side is finite as $R\to\infty.$  Thus  we prove our claim.
Immediately, for $\alpha>1$, one has 
\begin{equation}\label{V_sup =0 for alpha >1}
	\mathcal{V}_{sup}(e^{n\mathcal{L}(f)})=0.
\end{equation}

On the other hand, if $\alpha\leq 1$, for any small  $\epsilon\in (0,1)$, with help of Lemma \ref{lem: e^nLf on B_R+1 B_R-1} again, there exists $R_\epsilon>1$ such that for any $R\geq R_\epsilon$, one has
	\begin{equation}\label{equ e^nLf B_R+1 minu B_R-1}
		 \int_{B_{R+1}(0)\backslash B_{R-1}(0)}e^{n\mathcal{L}(f)}\ud x\leq C R^{n-1-n\alpha +\epsilon}.
	\end{equation}
	Similarly as before, there holds
	\begin{align*}
		\int_{B_R(0)}e^{n\mathcal{L}(f)}\ud x\leq &	\int_{B_{R_\epsilon}(0)}e^{n\mathcal{L}(f)}\ud x+\sum_{i=[R_\epsilon]+1}^{[R]}\int_{B_{i+1}(0)\backslash B_{i-1}(0)}e^{n\mathcal{L}(f)}\ud x.\\
		\leq &C(\epsilon)+C\sum_{i=[R_\epsilon]+1}^{[R]}C i^{n-1-n\alpha+\epsilon}\\
	\leq 	&C(\epsilon)+C\int_{[R_\epsilon]}^{[R]+1}t^{n-1-n\alpha+\epsilon}\ud t\\
	\leq&C(\epsilon)+C\frac{1}{n-n\alpha+\epsilon}\left(([R]+1)^{n-n\alpha+\epsilon}-([R_\epsilon]+1)^{n-n\alpha+\epsilon}\right)\\
	\leq & C(\epsilon)+C(\epsilon)(R+1)^{n-n\alpha+\epsilon}.
	\end{align*}
By letting $R\to\infty$, there holds
\begin{equation}\label{lambda enLf upper bound}
	\mathcal{V}_{sup}(e^{n\mathcal{L}(f)})\leq 1-\alpha+\frac{\epsilon}{n}.
\end{equation}
Due to the arbitrary choice of $\epsilon$ and \eqref{V_sup =0 for alpha >1}, we have
\begin{equation}\label{lambda sup}
	\mathcal{V}_{sup}(e^{n\mathcal{L}(f)})\leq (1-\alpha)^+.
\end{equation}
Combining \eqref{lambda inf} with \eqref{lambda sup}, one has
$$\mathcal{V}(e^{n\mathcal{L}(f)})=(1-\alpha)^+.$$

\end{proof}

The above lemmas provide integral results concerning the asymptotic behavior of $\mathcal{L}(f)$. If we impose additional conditions on $f$, we can obtain some pointwise estimates that are of independent interest similar to the two-dimensional case in \cite{CL93} and \cite{CLin}.

\begin{lemma}\label{lem:Q^+ and Q^-} 
	If  $f^+$ has compact support, there holds 
	$$ \mathcal{L}(f)(x)\leq -\alpha \log|x|+C,\quad |x|\gg1.$$
Conversely, if $f^-$ has compact support, there holds
	$$\mathcal{L}(f)(x)\geq -\alpha\log|x|-C,\quad |x|\gg1.$$
\end{lemma}
\begin{proof}
	By direct computation, we have
	\begin{align*}
		&\frac{(n-1)!|\mathbb{S}^n|}{2}(\mathcal{L}(f)(x)+\alpha\log|x|)\\
		=&\int_{\mr^n}\log\frac{|x|\cdot(|y|+1)}{|x-y|}f\ud y+\int_{\mr^n}\log\frac{|y|}{|y|+1}f\ud y\\
		=&\int_{\mr^n}\log\frac{|x|\cdot(|y|+1)}{|x-y|}f^+\ud y-\int_{\mr^n}\log\frac{|x|\cdot(|y|+1)}{|x-y|}f^-\ud y+C.
	\end{align*}
For $|x|\geq 1 $, it is easy to check that 
$$\frac{|x|\cdot(|y|+1)}{|x-y|}\geq 1$$
which shows that 
$$\log\frac{|x|\cdot(|y|+1)}{|x-y|}\geq 0.$$
If $f^+$ has compact support, there exists  $R_1>0$ such that $supp(f^+)\subset B_{R_1}(0)$.
And for $|x|\geq 2R_1$, we have
\begin{align*}
	&\int_{\mr^n}\log\frac{|x|\cdot(|y|+1)}{|x-y|}f^+\ud y\\
	=&\int_{B_{R_1}(0)}\log\frac{|x|\cdot(|y|+1)}{|x-y|}f^+\ud y\\
	\leq &\log\left(2|R_1|+2\right)\int_{B_{R_1}(0)}f^+\ud y\\
	\leq &C
\end{align*}
which yields that 
$$\mathcal{L}(f)(x)\leq -\alpha\log|x|+C,\quad |x|\gg1.$$
Reversely, if $f^-$ has compact support, following the same argument, there holds
$$\int_{\mr^n}\log\frac{|x|\cdot(|y|+1)}{|x-y|}f^-\ud y\leq C,\quad |x|\gg1$$ and then
$$\mathcal{L}(f)(x)\geq -\alpha\log|x|-C,\quad |x|\gg1.$$
\end{proof}

Interestingly, $\mathcal{L}(f)$ exhibits integral estimates for derivatives that are similar to the estimates for Green's representation on compact manifolds (see Lemma 2.3 in \cite{Malchiodi}). The following lemma plays an important role in the proof of the necessity for being  normal solutions.
\begin{lemma}\label{lem: Delta Lf}
For $R\gg1$, there holds
	$$\int_{B_R(0)}|\mathcal{L}(f)|\ud x=O((\log R)\cdot R^n).$$
	In particular, for $n\geq 4$, one has  
	$$\int_{B_R(0)}|\Delta \mathcal{L}(f)|\ud x=O(R^{n-2}),\quad \int_{B_R(0)}|\nabla \mathcal{L}(f)|^2\ud x=O(R^{n-2}).$$
\end{lemma}
\begin{proof}
A direct computation and 	 Fubini's theorem yield that
	\begin{align*}
		&\int_{B_R(0)}|\mathcal{L}(f)|\ud x\\
		\leq &C\int_{B_R(0)}\int_{\mr^n\backslash B_{2R}(0)}|\log\frac{ |y|}{|x-y|}|\cdot|f(y)|\ud y\ud x\\
		&+C\int_{B_R(0)}\int_{ B_{2R}(0)}|\log\frac{ |y|}{|x-y|}|\cdot|f(y)|\ud y\ud x\\
		\leq & CR^n+CR^n\int_{B_1(0)}|\log|y||\cdot|f(y)|\ud y\\
		&+CR^n\int_{B_{2R}(0)\backslash B_1(0)}|\log|y||\cdot |f(y)|\ud y\\
		&+\int_{B_{2R}(0)} |f(y)| \int_{B_{3R}(0)}|\log|z||\ud z\\
		\leq &CR^n+C(\log R)R^n\\
		\leq &C(\log R)R^n.
	\end{align*}
	Similarly, for $n\geq 4$, applying  Fubini's theorem again, we have 
	\begin{align*}
		&\int_{B_R(0)}|\Delta\mathcal{L}(f) |\ud x\\
		\leq &C\int_{B_R(0)}\int_{\mr^n}\frac{1}{|x-y|^2}|f(y)|d y\ud x\\
		\leq &C\int_{B_R(0)}\int_{\mr^n\backslash B_{2R}(0)}\frac{1}{|x-y|^2}|f(y)|\ud y\ud x\\
		&+C\int_{B_R(0)}\int_{B_{2R}(0)}\frac{1}{|x-y|^2}|f(y)|\ud y\ud x\\
		\leq &CR^{n-2}+C\int_{B_{2R}(0)}|f(y)|\ud y\int_{B_{3R}(0)}\frac{1}{|x|^2}\ud x\\
		\leq &CR^{n-2}.
	\end{align*}
Using H\"older's inequality and the assumption  $f\in L^1(\mr^n)$, 
there holds
\begin{align*}
	\int_{B_R(0)}|\nabla \mathcal{L}(f)|^2\ud x=&C\int_{B_R(0)}|\int_{\mr^n}\frac{x-y}{|x-y|^2}f(y)\ud y|^2\ud x\\
	\leq & C\int_{B_R(0)}\left(\int_{\mr^n}\frac{1}{|x-y|^2}|f(y)|\ud y\int_{\mr^n}|f(y)|\ud y\right)\ud x\\
	\leq &C\int_{B_R(0)}\int_{\mr^n}\frac{1}{|x-y|^2}|f(y)|\ud y\ud x\\
	\leq &C R^{n-2}.
\end{align*}
\end{proof}
The following Liouville-type theorem may be well-known to experts. For reader's convenience, we would like to  give the  proof by slightly modifying Theorem 5 in \cite{Mar MZ}.
\begin{lemma}\label{lem: polynomial growth is polynomial}
	Foy any integer $m\geq 1,k\geq 0$, suppose that  $(-\Delta)^m\varphi(x)=0$ and 
	$$\int_{B_R(0)}\varphi^+\ud x=o(R^{k+n+1}).$$
	Then $\varphi(x)$ is a polynomial with 
	$\deg\varphi\leq \max\{2m-2,k\}.$
\end{lemma}

\begin{proof}
	With help of Proposition 4 in \cite{Mar MZ}  (or Theorem 2.2.7, p. 29 in \cite{Evans}) as well as Pizzetti's formula (See \cite{Piz} or Lemma 3 in \cite{Mar MZ}) for polyharmonic functions, for any $x\in \mr^n$ and $R> |x|$, we have
	\begin{align*}
		|D^{l+1}\varphi(x)|
		\leq& \frac{C}{R^{l+1}}\frac{1}{|B_R(x)|}\int_{B_R(x)}|\varphi|\ud y\\
			\leq &\frac{C}{R^{l+1}}\frac{2^n}{|B_{2R}(0)|}\int_{B_{2R}(0)}|\varphi|\ud y\\
		\leq & \frac{C\cdot 2^n}{R^{l+1}}\frac{1}{|B_{2R}(0)|}\int_{B_{2R}(0)}(-\varphi+2\varphi^+)\ud y\\
		=&\frac{C\cdot 2^n}{R^{l+1}}\left(-\sum^{m-1}_{i=0}c_i(2R)^{2i}\Delta^i\varphi(0)\right)+\frac{C\cdot 2^n}{R^{l+1}}\frac{1}{|B_{2R}(0)|}\int_{B_{2R}(0)}2\varphi^+\ud y\\
		\leq & CR^{2m-2-l-1}+o(R^{k-l})
\end{align*}
where $c_i$ are positive constants.
If $k>2m-2$, choose $l=k$ and  otherwise choose $l=2m-2$.
By letting $R\to\infty$, we show that $\varphi(x)$ must be a polynomial of degree at most $l$. Thus we finish our proof.
\end{proof}

With above preparations, we will give the proof of  Theorem \ref{thm:decomposition}.

{\bf Proof of Theorem \ref{thm:decomposition}:}
\begin{proof}
		 Set $P(x):=w(x)-\mathcal{L}(f)$. By using Green's function \eqref{Green's function}, we obtain 
		\begin{equation}\label{polyharmonic P}
			(-\Delta)^{\frac{n}{2}}P(x)=0.
		\end{equation} 
	
	Firstly, we show that $P(x)$ is a polynomial function.  
 Since $w(x)$ exhibits polynomial volume growth, we can find $s\geq 0$ such that
 \begin{equation}\label{poly volume growth}
 	\int_{B_R(0)}e^{nw}\ud x=O(R^s).
 \end{equation} 
 Let $q\in (0,1)$ be a small number that satisfies $(\frac{s}{n}-1)q\leq \frac{1}{2}.$
	Then by using H\"older's inequality and Lemma \ref{lem: Delta Lf}, we have  
	\begin{align*}
		\int_{B_R(0)}P^+\ud x
		\leq &	\int_{B_R(0)}w^+\ud x+	\int_{B_R(0)}|\mathcal{L}(f)|\ud x\\
		\leq &\int_{B_R(0)}\frac{1}{q}e^{qw}\ud x+C(\log R)R^n\\
		\leq &\frac{1}{q}(\int_{B_R(0)}e^{nu}\ud x)^{\frac{q}{n}}(\int_{B_R(0)}\ud x)^{1-\frac{q}{n}}+C(\log R)R^n\\
		\leq &CR^{\frac{sq}{n}+n-q}+C(\log R)R^n\\
		\leq &CR^{n+\frac{1}{2}}
	\end{align*}
where we have used the fact $qw^+\leq e^{qw}.$
	Lemma \ref{lem: polynomial growth is polynomial} concludes that $P(x)$ is a polynomial function with a degree no greater than $n-2$.

 When $n=2$, we immediately have $P(x)\equiv C.$ As for the case $n\geq 4$,
 we will show that $P(x)$ has a finite upper bound.  We   argue by contradiction.
	Inspired by the argument of Lemma 11 in \cite{Mar MZ}, we define $$\psi(r):=\sup_{\partial B_r(0)} P.$$
	If $\sup_{\mr^n}P=+\infty$, using Theorem 3.1 in \cite{Gor}, there exists $s_1>0$ such that 
	\begin{equation}\label{psi(r)/r^{s_1}}
		\lim_{r\to\infty}\frac{\psi(r)}{r^{s_1}}=+\infty.
	\end{equation}
Since $P(x)$ is a polynomial of degree at most $n-2$, there holds
	$|\nabla P(x)|\leq C|x|^{n-3}$ for $|x|\gg1$.  There exist $R_1>0$ such that for any $r\geq R_1$, we can find $x_r$ with $|x_r|=r$ such that
	$$P(x)\geq r^{s_1}, \mathrm{for}\; |x-x_r|\leq r^{3-n}.$$
	Then with help of Jensen's inequality and Lemma \ref{lem:B_x^-r_1}, for $R\geq R_1$, we have 
	\begin{align*}
		\int_{B_{2R}(0)}e^{nw}\ud x\geq&\int_{B_{R^{3-n}}(x_R)}e^{nw}\ud x\\
		\geq &e^{nR^{s_1}} \int_{B_{R^{3-n}}(x_R)}e^{n\mathcal{L}(f)}\ud x\\
		\geq &e^{nR^{s_1}}|B_{R^{3-n}}(x_R)|\exp\left(\frac{1}{|B_{R^{3-n}}(x_R)|}\int_{B_{R^{3-n}}(x_R)}n\mathcal{L}(f)\ud x \right)\\
		\geq & |\mathbb{S}^n|e^{nR^{s_1}}R^{n(3-n)}\exp(-C\log R)\\
		= &|\mathbb{S}^n|e^{nR^{s_1}-n(n-3)\log R- C\log R}
	\end{align*}
	which contradicts to \eqref{poly volume growth}. Thus we obtain our desired result
	\begin{equation}\label{Pleq C}
		\sup_{\mr^n}P(x)<+\infty.
	\end{equation}
	
Finally, we discuss the necessary and sufficient conditions for $P(x)$ to be a constant for $n\geq 4$.
 On one hand,  we consider the case (a).
	If $P(x)\equiv C$, by Lemma \ref{lem: Delta Lf}, we have
	$$  \int_{B_R(0)}|\Delta w|\ud x=\int_{B_R(0)}|\Delta\mathcal{L}(f)|\ud x=O(R^{n-2})=o(R^{n}).$$ 
	Conversely, if \eqref{Delta w=o(R^n)} holds, due to Lemma \ref{lem: Delta Lf}, there holds
	\begin{equation}\label{Delta P}
		\int_{B_R(0)}|\Delta P|\ud x\leq \int_{B_R(0)}\left(|\Delta w|+|\Delta\mathcal{L}(f)|\right)\ud x=o(R^{n}).
	\end{equation}
We have $\Delta P=0$ since $\Delta P$ is also a polynomial. Due to  the upper bound \eqref{Pleq C}, Liouville's theorem yields that $P(x)$ must be a constant.

	On the other hand,  we deal with the case (b).  When $P(x)$ is a constant, due to Lemma \ref{lem: Delta Lf}, there holds
	$$\int_{B_R(0)}|w|\ud x\leq \int_{B_R(0)}(|\mathcal{L}(f)|+C)\ud x=O(\log R \cdot R^n)+O(R^n)=o(R^{n+2}).$$ Conversely, if \eqref{w=o(R^n+2)} holds, with help of Lemma \ref{lem: Delta Lf} again, one has 
	$$\int_{B_R(0)}|P(x)|\ud x\leq \int_{B_R(0)}(|w|+|\mathcal{L}(f)|)\ud x=o(R^{n+2}).$$
Then, due to $P(x)$ is polynomial and the upper bound  \eqref{Pleq C}, we must have $P(x)\equiv C$.
 
	Finally,  we finish our proof.
\end{proof}

In fact, with help of a Liouville type theorem Lemma \ref{lem: polynomial growth is polynomial}, we will get a more general decomposition theorem as follows.
\begin{theorem}
	Let $n\geq 2$ be an even integer and consider a smooth solution $w(x)$ to \eqref{w equation} with smooth $f\in L^1(\mathbb{R}^n)$. Suppose that for some integer $k\geq0$,
	\begin{equation}\label{condtion for w^+}
		\fint_{B_R(0)}w^+\ud x=o(R^{k+1}).
	\end{equation}
	Then 
	$$w(x)=\mathcal{L}(f)+P(x)$$
	where $P(x)$ is a polynomial with $\deg(P)\leq \max\{k,n-2
	\}$.
\end{theorem}
\begin{proof}
	Set $P(x)=w(x)-\mathcal{L}(f)$ and we have
	$$(-\Delta)^{\frac{n}{2}}P(x)=0.$$
	With help of Lemma \ref{lem: Delta Lf} and the assumption \eqref{condtion for w^+}, there holds
	\begin{align*}
		\fint_{B_R(0)}P^+\ud x\leq &\fint_{B_R(0)}w^+\ud x+\fint_{B_R(0)}|\mathcal{L}(f)|\ud x\\
		\leq &o(R^{k+1})+C\log R\\
		=&o(R^{k+1}).
	\end{align*}
Applying Lemma \ref{lem: polynomial growth is polynomial}, one has $P(x)$ is a polynomial and $\deg(P)\leq \max\{k,n-2\}$.
\end{proof}

With help of the  decomposition in Theorem \ref{thm:decomposition}, we can establish the following volume comparison theorem.
\begin{theorem}\label{thm: lambda leq 1-alpha}
	Consider a smooth solution $w(x)$ to \eqref{w equation}, where $n\geq 2$ is an even integer with smooth $f\in L^1(\mr^n)$.  
	\begin{enumerate}[(a)]
		\item Assume that $\mathcal{V}_{sup}(e^{nw})<+\infty$.
		Then we have $
		\mathcal{V}_{\sup}(e^{nw})\leq (1-\alpha)^+.$
		If  $\alpha>1$ in addition, the volume is finite i.e. $e^{nw}\in L^1(\mr^n)$.
		\item If $w(x)$ is a normal solution, then the limit of volume growth exists and 
		\begin{equation}\label{V_e^nw=1-alpha_0 in Theorem }
			\mathcal{V}(e^{nw})=(1-\alpha)^+.
		\end{equation}
	\end{enumerate}
	
\end{theorem}
\begin{proof}
If $\mathcal{V}_{sup}(e^{nw})<+\infty$, with help of Theorem \ref{thm:decomposition}, we have the decomposition
	\begin{equation}\label{w=Lf+P}
		w(x)=\mathcal{L}(f)(x)+P(x)
	\end{equation}
	where $P(x)$ is a polynomial and $P(x)\leq C.$
	Then
	\begin{equation}\label{EQ:e^nu leq e^nLQEnu}
		\int_{B_R(0)}e^{nw}\ud x= \int_{B_R(0)}e^{n\mathcal{L}(f)(x)+nP(x)}\ud x\leq C\int_{B_R(0)}e^{n\mathcal{L}(f)}\ud x
	\end{equation}
	which yields that
	\begin{equation}\label{lambda e^nw leq enLf}
		\mathcal{V}_{sup}(e^{nw})\leq \mathcal{V}_{sup}(e^{n\mathcal{L}(f)}).
	\end{equation}
	Applying Lemma \ref{lem: e^nLf B_R} and the estimate  \eqref{lambda e^nw leq enLf}, one has
	$$\mathcal{V}_{sup}(e^{nw})\leq  (1-\alpha)^+.$$
	If $\alpha>1,$ Lemma \ref{lem: e^nLf B_R} combing with \eqref{EQ:e^nu leq e^nLQEnu} deduces that
	$$\int_{\mr^n}e^{nw}\ud x<+\infty.$$ 
	
	Specially, if  $w(x)$ is a normal solution, there holds
	$$w(x)=\mathcal{L}(f)(x)+C.$$
	Due to Lemma \ref{lem: e^nLf B_R},  then the limit of volume growth exists and 
	$$\mathcal{V}(e^{nw})=(1-\alpha)^+.$$

	Thus we finish our proof.
\end{proof}
\begin{remark}
	It is easy to see from \eqref{V_e^nw=1-alpha_0 in Theorem } that the volume is infinite for normal solutions with $\alpha<1$. However, the case $\alpha=1$ is more delicate. It is possible for both infinite and finite volume to occur. See Section \ref{section: example}.
\end{remark} 

Until now, we haven't considered the assumption on the completeness of the metric. It is also interesting to note that if we replace the completeness assumption with the assumption of finite volume, we obtain a reversed Cohn-Vossen inequality.
For reader's convenience, for $n=2$, we rewrite \eqref{Q-curvature } as 
\begin{equation}\label{Gaussian equation}
	-\Delta u=Ke^{2u}
\end{equation}
where both $K$ and $u$ are smooth functions on $\mr^2$.
\begin{theorem}\label{thm: reversed Cohn-Voseen}
	Consider the equation \eqref{Gaussian equation}.
	If the volume is finite and the the negative part of Gaussian curvature is integrable i.e.
	$$\int_{\mr^2}e^{2u}\ud x<+\infty
	\quad \mathrm{and} \quad \int_{\mr^2}K^-e^{2u}\ud x <+\infty,$$
	then there holds
	\begin{equation}\label{K e^2u geq 2pi}
		\int_{\mr^2}Ke^{2u}\ud x\geq 2\pi.
	\end{equation}
\end{theorem}
\begin{proof}
If the integral $\int_{\mathbb{R}^2} K^+ e^{2u} \ud x = \infty$, the inequality \eqref{K e^2u geq 2pi} is automatically satisfied due to the condition $\int_{\mathbb{R}^2} K^- e^{2u} \ud x < \infty$. Therefore, we only need to consider the case where $\int_{\mathbb{R}^2} K^+ e^{2u} \ud x < \infty$, which implies that the total Gaussian curvature is finite. Under the assumption of finite volume, the volume entropy $\tau(g)$ is zero. According to Theorem \ref{thm:decomposition}, the solution $u$ must be a normal solution. By utilizing Theorem \ref{thm: lambda leq 1-alpha}, we obtain
	$$\int_{\mr^2}Ke^{2u}\ud x\geq 2\pi.$$
\end{proof}
\begin{remark}
	Actually, for the two-dimensional case, Lytchak provided a geometric proof (Theorem 1.6 in \cite{Ly}) and as he expected  in Remark 4.1 of \cite{Ly}, we  present an analytical proof based on the works of Huber \cite{Hu} here. 
\end{remark}

For higher-order cases, the finite volume assumption alone is insufficient to establish a lower bound on the total Q-curvature (refer to \cite{Chang-Chen}). In order to obtain a similar result to Theorem \ref{thm: reversed Cohn-Voseen}, additional restrictions on $\Delta u$ are required to ensure that the solution is normal, utilizing Theorem \ref{thm:decomposition}.
\begin{theorem}\label{thm: reversed Cohn-Voseen inequality for Q-curvature}
	Consider the equation \eqref{Q-curvature } with $n\geq 4$ is an even integer. Suppose that 
	$$\int_{\mr^n}e^{nu}\ud x<+\infty,\quad \int_{\mr^n}Q_g^-e^{nu}\ud x<+\infty$$
	and 
	\begin{equation}\label{Delta u behavior}
		\int_{B_R(0)}|\Delta u|\ud x=o(R^n).
	\end{equation}
	Then there holds
	$$\int_{\mr^n}Q_ge^{nu}\ud x\geq \frac{(n-1)!|\mathbb{S}^n|}{2}.$$
\end{theorem}
\begin{proof}
The proof is essentially the same as that of Theorem \ref{thm: reversed Cohn-Voseen}, and therefore we omit it here.
\end{proof}

\section{Strong $A_\infty$ weight  and length comparison}\label{3}

Consider the measure distance
$$\delta(x,y)=\left(\int_{B_{x,y}}e^{nu}\ud z\right)^{\frac{1}{n}}$$
where
$B_{x,y}$ is the Euclidean  ball in $\mr^n$ with diameter $|x-y|$ that contains $x$ and $y$ i.e.
$$B_{x,y}=B_{\frac{|x-y|}{2}}(\frac{x+y}{2}).$$
For a positive locally integrable  function $\varphi(x)$ on $\mr^n$, we say $\varphi(x)$ is an  $A_\infty$ weight if for all balls $B$
there holds
\begin{equation}\label{A_infty}
	\frac{1}{|B|}\int_{B}\varphi\ud x\leq C\exp\left(\frac{1}{|B|}\int_{B}\log\varphi\ud x\right).
\end{equation}
More equivalent definitions of $A_\infty$ weight can be found in \cite{Se}.  The definition of $A_\infty$ weight \eqref{A_infty}  taken  here was firstly introduced by Hru\v s\v cev  in \cite{Hru}. 
	 If we suppose the volume term $e^{nu}$ is an $A_\infty$ weight, 
	an important property of $A_\infty$ weight (See Proposition 3.12 in \cite{Se} or Lemma 3.5 in \cite{ACT}) shows that
	\begin{equation}\label{d_gleq delta_g}
	d_g(x,p)\leq C\delta(x,p)
	\end{equation}
where $C$ is a constant  independent of $x$ and $p$.
	In \cite{DS}, David and Semmes introduced a strong $A_\infty$ weight which is an $A_\infty$ weight additionally satisfying 
	\begin{equation}\label{strong A_infty}
		d_g(x,p)\geq C\delta(x,p).
	\end{equation}
	Briefly, for a strong $A_\infty$ weight, the measure distance and geodesic distance are equivalent. In harmonic analysis, $A_p$ weight has many interesting properties and we refer the interested reader to Chapter \uppercase\expandafter{\romannumeral5} of Stein's well-known monograph \cite{Stein}.
	With help of the property of strong  $A_\infty$ weight, we can  generalize a result of Li and Tam (See Corollary 3.3 in \cite{LT}).

\begin{lemma}\label{lem: inf V-g/V_0}
	Suppose  that $g=e^{2u}|dx|^2$ is a normal metric on $\mr^n$ where $n\geq2 $ is an even integer and  $u(x)$ satisfies  \eqref{Q-curvature } with $Q_ge^{nu}\in L^1(\mr^n)$, for fixed $p\in\mr^n$,
	there holds
	$$\lim_{R\to\infty}\frac{\log V_g(B_R(p))}{\log |B_R(p)|}=\lim_{|x|\to\infty} \frac{\log\delta(x,p)}{\log|x-p|}= (1-\alpha_0)^+.$$
\end{lemma}
\begin{proof}
	Since $u(x)$ is a normal solution, 
we have the decomposition using logrithmic potential: 
	$$u(x)=\mathcal{L}(Q_ge^{nu})(x)+C.$$
	With help of Theorem \ref{thm: lambda leq 1-alpha}, there holds
	\begin{equation}\label{e^nu on B_R(0)}
		\lim_{R\to\infty}\frac{\log V_g(B_R(0))}{\log |B_R(0)|}=(1-\alpha_0)^+.
	\end{equation}
	For each fixed $p\in\mr^n$ and  $R\gg1$, we have
	$$\int_{B_{R/2}(0)}e^{nu}\ud x\leq \int_{B_R(p)}e^{nu}\ud x\leq \int_{B_{2R}(0)}e^{nu}\ud x$$
	which yields that
	$$\lim_{R\to\infty}\frac{\log V_g(B_R(p))}{\log |B_R(p)|}=(1-\alpha_0)^+.$$

For $|x|\gg1$, it is not hard to check that 
$$B_{\frac{1}{8}|x|}(\frac{x}{2})\subset B_{x,p}=B_{\frac{|x-p|}{2}}(\frac{x+p}{2}).$$
Then, with help of Jensen's inequality and Lemma \ref{lem:B_r_0|x|}, it follows that
\begin{align*}
	\log\delta(x,p)\geq &\frac{1}{n}\log\int_{B_{\frac{1}{8}|x|}(\frac{x}{2})}e^{nu}\ud y\\
	=&\frac{1}{n}\log\fint_{B_{\frac{1}{8}|x|}(\frac{x}{2})}e^{nu}\ud y+\log|x|+\log\frac{|\mathbb{S}^n|^{\frac{1}{n}}}{8}\\
	\geq &\fint_{B_{\frac{1}{8}|x|}(\frac{x}{2})}u\ud y+\log|x|+C\\
	\geq &(1-\alpha_0+o(1))\log|x|+C
\end{align*}
which yields that 
	 $$
\lim_{|x|\to\infty}\inf \frac{\log\delta(x,p)}{\log|x-p|}\geq 1-\alpha_0.$$
It is not hard to check $|B_1(p)\cap B_{x,p}|\geq\frac{|\mathbb{S}^n|}{4}$ for $|x|\gg1$ and then
$$ \lim_{|x|\to\infty}\inf \frac{\log\delta(x,p)}{\log|x-p|}\geq0$$
which yields that 
\begin{equation}\label{inf delta_g/delta_0}
 \lim_{|x|\to\infty}\inf \frac{\log\delta(x,p)}{\log|x-p|}\geq (1-\alpha_0)^+.
\end{equation}
For $|x|\gg1$, one may easily check that 
$$\delta(x,p)^n\leq\int_{B_{2|x|}(0)}e^{nu}\ud y.$$
Apply  the estimate \eqref{e^nu on B_R(0)} to get
$$\lim_{|x|\to\infty}\sup\frac{\log\delta(x,p)}{\log|x-p|}\leq (1-\alpha_0)^+.$$ Combing with \eqref{inf delta_g/delta_0}, one has
$$\lim_{|x|\to\infty}\frac{\log\delta(x,p)}{\log|x-p|}=(1-\alpha_0)^+.$$

\end{proof}

Now, we are going to give the prove of Theorem \ref{thm: length comparison identity} and show some examples in this section.

{\bf Proof of Theorem \ref{thm: length comparison identity}:}
\begin{proof}
	
	Choose the curve $\gamma(t)=t\frac{x}{|x|}$ and the point $p_1=|p|\frac{x}{|x|}$. Then one has 
	$$d_g(x,p)\leq d_g(p_1,p)+d_g(x,p_1)\leq C+ \int_{|p|}^{|x|}e^{u(\gamma(t))}\ud t.$$
	To estimate the second term, we firstly claim that for $|x|\gg1$, there holds
	\begin{equation}\label{cliam for distance}
		\int^{|x|+\frac{1}{2}}_{|x|-\frac{1}{2}}e^{u(\gamma(t))}\ud t\leq |x|^{-\alpha_0+o(1)}.
	\end{equation}
	Using Lemma \ref{lem: L(f)} and the assumption $u(z)$ is normal, we have
	\begin{equation}\label{equ:u leq Q^+_g}
		u(z)\leq(-\alpha_0+o(1))\log|z|+\frac{2}{(n-1)!|\mathbb{S}^n|}\int_{B_{1}(z)}\log\frac{1}{|z-y|}Q^+_g(y)e^{nu(y)}\ud y.
	\end{equation}
	By slightly modifying Lemma \ref{lem: L(f)},
	we can easily get
	$$u(z)\leq (-\alpha_0+o(1))\log|z|+\frac{2}{(n-1)!|\mathbb{S}^n|}\int_{B_{1/4}(z)}\log\frac{1}{|z-y|}Q_g^+(y)e^{nu(y)}\ud y.$$
	If $Q_g^+=0$ a.e. on $B_1(x)$, for   any $z\in B_{1/2}(x)$ ,
	there holds 
	$$u(z)\leq (-\alpha_0+o(1))\log|x|.$$
	Then the claim follows in this case.
	Otherwise, since $Q_ge^{nu}\in L^1(\mr^n)$, there exists $R_1>0$ such that for any $|x|>R_1$, 
	$$\int_{B_1(x)}Q_g^+(y)e^{nu(y)}\ud y\leq \frac{(n-1)!|\mathbb{S}^n|}{4}.$$ 
	The estimate \eqref{equ:u leq Q^+_g} and Jensen's inequality yield that for $|x|>R_1$ and $z\in B_{1/2}(x)$
	\begin{align*}
		e^{u(z)}= &Ce^{\mathcal{L}(Q_ge^{nu})}\\
		\leq &|x|^{-\alpha_0+o(1)}\exp\left(\frac{2}{(n-1)!|\mathbb{S}^n|}\int_{B_1(x)}\log\frac{1}{|z-y|}Q_g^+(y)e^{nu}\ud y\right)\\
		\leq &|x|^{-\alpha_0+o(1)}\int_{B_1(x)}(\frac{1}{|z-y|})^{\frac{2\|Q_g^+e^{nu}\|_{L^1(B_1(x))}}{(n-1)!|\mathbb{S}^n|}}\frac{Q_g^+(y)e^{nu}}{\|Q_g^+e^{nu}\|_{L^1(B_1(x))}}\ud y\\
		\leq &|x|^{-\alpha_0+o(1)}\int_{B_1(x)}\frac{1}{\sqrt{|z-y|}}\frac{Q_g^+(y)e^{nu}}{\|Q_g^+e^{nu}\|_{L^1(B_1(x))}}\ud y.
	\end{align*}
	Then by using Fubini's theorem and the above estimate, we have
	\begin{align*}
		\int^{|x|+\frac{1}{2}}_{|x|-\frac{1}{2}}e^{u(\gamma(t))}\ud t\leq &|x|^{-\alpha_0+o(1)}\int^{|x|+\frac{1}{2}}_{|x|-\frac{1}{2}}\int_{B_1(x)}\frac{1}{\sqrt{|\gamma(t)-y|}}\frac{Q_g^+(y)e^{nu(y)}}{\|Q_g^+e^{nu}\|_{L^1(B_1(x))}}\ud y\ud t\\
		\leq& C|x|^{-\alpha_0+o(1)}\\
		=&|x|^{-\alpha_0+o(1)}
	\end{align*}
	Thus we prove our claim \eqref{cliam for distance}.
	
	For any $\epsilon>0$, the estimate \eqref{cliam for distance} shows that there exist $R_2>0$ such that for $|x|\geq R_2$
	$$\int_{|x|-\frac{1}{2}}^{|x|+\frac{1}{2}}e^{u(\gamma(t))}\leq |x|^{-\alpha_0+\epsilon}.$$
	Then we have
	\begin{equation}\label{d_g(x,p)leq C+i^-alpha_0}
		d_g(x,p)\leq C+\int_{|p|}^{|x|}e^{u(\gamma(t))}\ud t\leq C(\epsilon)+\sum^{[|x|]+1}_{i=[R_2]+2}i^{-\alpha_0+\epsilon}.
	\end{equation}

	If $\alpha_0>1$, by choosing small $\epsilon$ and using \eqref{d_g(x,p)leq C+i^-alpha_0}, one has 
	\begin{equation}\label{d_g finite}
		\lim_{|x|\to\infty}\sup d_g(x,p)<+\infty.
	\end{equation}
	Moreover, for any $x,y\in\mr^n$, considering a fixed point $p$,  notice that
	$d_g(x,y)\leq d_g(x,p)+d_g(p,y)$
	to get 
	$$\mathrm{diam}(\mr^n,e^{2u}|dx|^2)<+\infty.$$
	Since, for $|x|\gg1$,  any curves connecting the points  $x$ and $p$ must cross the sphere $\partial B_1(p)$,
	then 
	$$d_g(x,p)\geq \inf_{|y-p|=1}d_g(y,p)>0$$
	which shows that  
	\begin{equation}\label{d_g has low bound}
		\lim_{|x|\to\infty}\inf\frac{\log d_g(x,p)}{\log|x-p|}\geq 0.
	\end{equation}
	Obviously,  combing \eqref{d_g finite} with \eqref{d_g has low bound}, one has
	$$\lim_{|x|\to\infty}\frac{\log d_g(x,p)}{\log|x-p|}=0.$$
	
	Now, we will  deal with the case $\alpha_0\leq 1$.
	From \eqref{d_g(x,p)leq C+i^-alpha_0}, due to the monotonicity of $s^{-\alpha_0+\epsilon}$, one has 
	$$d_g(x,p)\leq C(\epsilon)+\int^{[|x|]+2}_{[R_2]+1}s^{-\alpha_0+\epsilon}\ud s\leq  C(\epsilon)+\frac{([|x|]+2)^{1-\alpha_0+\epsilon}}{1-\alpha_0+\epsilon}$$
	which yields that
	$$\lim_{|x|\to\infty}\sup\frac{\log d_g(x,p)}{\log|x-p|}\leq 1-\alpha_0+\epsilon.$$
	By the arbitrary choice of $\epsilon$, we obtain 
	\begin{equation}\label{d_g upper bound}
		\lim_{|x|\to\infty}\sup\frac{\log d_g(x,p)}{\log|x-p|}\leq 1-\alpha_0.
	\end{equation}
	Thus for the case $\alpha_0 =1$, combing with \eqref{d_g has low bound},  one has
	$$\lim_{|x|\to\infty}\frac{\log d_g(x,p)}{\log|x-p|}=0.$$
	If  $\alpha_0<1$, it was shown in Corollary 1.7 of \cite{Wang IMRN} that $e^{nu}$ is a strong $A_\infty$ weight which deduces that 
	\begin{equation}\label{strong A_infty weight property}
		C^{-1}\delta (x,p)\leq d_g(x,p)\leq C\delta(x,p).
	\end{equation}
	Using Lemma \ref{lem: inf V-g/V_0}, one has
	$$\lim_{|x|\to\infty}\frac{\log d_g(x,p)}{\log|x-p|}=1-\alpha_0.$$
	
	Finally, we finish our proof.
	
\end{proof}

\begin{theorem}\label{thm:d_g(x,p) and |x-p|}
	Suppose that $g=e^{2u}|dx|^2$ is a complete and   normal metric. There holds
	\begin{equation}\label{d_g(x,y)/|x-y|}
		\lim_{R\to\infty}\frac{\log V_g(B_R(p))}{\log |B_R(p)|}=
		\lim_{|x|\to\infty}\frac{\log d_g(x,p)}{\log|x-p|}=1-\alpha_0.
	\end{equation}
\end{theorem}
\begin{proof}
	In fact, since $g$ is a complete and normal metric, Theorem 1.3  in \cite{CQY} (See also \cite{NX}, \cite{Fa}) has shown that
	$\alpha_0\leq 1.$  
	In fact, applying Theorem \ref{thm: length comparison identity}, it is not hard to show that $\alpha_0\leq 1$ because if  $g$ is complete, for fixed $p\in\mr^n$,  we have
	$d_g(x,p)\to\infty$ as $|x|\to\infty$.  Combing Lemma \ref{lem: inf V-g/V_0} with  Theorem \ref{thm: length comparison identity}, we finish our proof.
\end{proof}

\section{Necessary and sufficient conditions for normal solutions}\label{section: iff normal solution}

Now, return  to  the conformal invariant equation \eqref{Q-curvature } with finite total Q-curvature.
As we defined in Section \ref{2}, we say $u(x)$ is a normal solution to \eqref{Q-curvature } if $u(x)=\mathcal{L}(Q_ge^{nu})+C$
where    $C$ is a constant. 
Therefore, we need to outline a criterion to identify normal solutions. 
In \cite{CQY}, Chang, Qing and Yang demonstrated that solutions are normal when $R_g\geq 0$ near infinity. Similarly, Wang et al. showed in \cite{WW} that solutions are normal if $\int_{\mr^n}(R_g^-)^{\frac{n}{2}}e^{nu}\ud x<+\infty.$   In the following theorem, we want to generalize their results. Our proof is based on similar principles as theirs, with  emphasizing the necessity of being normal solutions, a previously overlooked consideration.

\begin{theorem}\label{thm: necessary and sufficient for normal solution}
	Consider a smooth function $u(x)$ that satisfies equation \eqref{Q-curvature } with even integer $n\geq4$ and $Q_ge^{nu}\in L^1(\mr^n)$. Then $u(x)$ is a normal solution if and only if the following condition  is satisfied:
	\begin{equation}\label{R_g^-=o(R^n)}
		\int_{B_R(0)}R_g^-e^{2u}\ud x=o(R^{n}).
	\end{equation}
\end{theorem}
\begin{proof}
	For brevity,  set
	$v(x)=\mathcal{L}(Q_ge^{nu}).$ With help of Green's function \eqref{Green's function}, one has 
	$$(-\Delta)^{n/2}(u-v)=0.$$
	Set $h:=u-v$ which  is a polyharmonic function satisfying
	\begin{equation}\label{polyharmonic h}
		(-\Delta)^{\frac{n}{2}}h=0.
	\end{equation} 
	Consider $R> 2|x|$ and then 
	Lemma \ref{lem: Delta Lf} yields that
	\begin{equation}\label{Delta v leq CR^n-2 }
		\int_{B_R(x)}|\Delta v(z)|\ud z\leq \int_{B_{2R}(0)}|\Delta v(z)|\ud z\leq CR^{n-2}
	\end{equation}
	as well as 
	\begin{equation}\label{nabla v^2}
		\int_{B_R(x)}|\nabla v|^2\ud z\leq\int_{B_{2R}(0)}|\nabla v|^2\ud z\leq  CR^{n-2}.
	\end{equation}
	On one hand, if  \eqref{R_g^-=o(R^n)}  holds, with help of  the definition $R_g$ \eqref{scalar curvature},  for $R>|x|$, one has 
	$$	\int_{B_R(x)}\Delta u\ud y+\frac{n-2}{2}\int_{B_R(x)}|\nabla u|^2\ud y\\
	\leq\frac{1}{2(n-1)}\int_{B_{2R}(0)}R_g^-e^{2u}\ud y= o(R^n).$$
	Based on the estimate $|\nabla  h|^2\leq 2|\nabla u|^2+2|\nabla v|^2$ and the estimates  \eqref{Delta v leq CR^n-2 }, \eqref{nabla v^2}, there holds
	\begin{equation}\label{Delta h+|nabla h|}
		\int_{B_R(x)}\Delta h\ud y+\frac{n-2}{4}\int_{B_R(x)}|\nabla h|^2\ud y\leq o(R^n).
	\end{equation}
	Immediately, we have
	\begin{equation}\label{Delta h leq oR^n}
		\int_{B_R(x)}\Delta h\ud y\leq o(R^n).
	\end{equation}
	With help of Pizzetti's formula (See \cite{Piz} or Lemma 3 in \cite{Mar MZ}) for polyharmonic functions, we have
	$$\frac{1}{|B_R(x)|}\int_{B_R(x)}\Delta h\ud y=\sum^{\frac{n}{2}-2}_{i=0}c_i R^{2i}\Delta^{i}(\Delta h(x)) $$
	where $c_i$ are positive constants.
	By letting $R\to\infty$, the estimate \eqref{Delta h leq oR^n} yields that the leading term 
	$$\Delta^{\frac{n}{2}-1}h(x)\leq 0.$$ 
	With help of Liouville's theorem for the harmonic function  $\Delta^{\frac{n}{2}-1}h$, we have
	\begin{equation}\label{Delta^n/2-1h}
		\Delta^{\frac{n}{2}-1}h=C_0\leq 0
	\end{equation}
	for some constant $C_0$.
	Consequently, for any $1\leq j\leq n$, one has 
	$$\Delta^{\frac{n}{2}-1}\partial_j h=0.$$
	When $n=4$, we have obtained 
	$$\Delta h=C_0\leq 0,\quad \Delta \partial_j h=0.$$ 
	For $n>4$, 
	applying  Pizzetti's formula to $\partial_j h$, 
	there holds
	$$\frac{1}{|B_R(x)|}\int_{B_R(x)}\partial_j h\ud y=\sum^{\frac{n}{2}-2}_{i=0}c_i R^{2i}\Delta^{i}\partial_jh(x).$$
	H\"older's inequality and the estimate \eqref{Delta h+|nabla h|} yield that 
	\begin{align*}
		&\frac{1}{|B_R(x)|}\int_{B_R(x)}\Delta h\ud y+\frac{n-2}{4}	(\frac{1}{|B_R(x)|}\int_{B_R(x)}\partial_j h\ud y)^2\\
		\leq &\frac{1}{|B_R(x)|}\int_{B_R(x)}\Delta h\ud y+\frac{n-2}{4}	\frac{1}{|B_R(x)|}\int_{B_R(x)}|\nabla h|^2\ud y\\
		\leq &o(1)
	\end{align*}
	where $o(1)\to 0$ as $R\to\infty$.
	Since $n>4$, the leading term of the left side is 
	$$(c_{\frac{n}{2}-2}(\Delta^{\frac{n}{2}-2}\partial_j h(x))^2R^{2n-8}$$
	which concludes that 
	$$\Delta^{\frac{n}{2}-2}\partial_jh= 0.$$
	Then one has  $	\Delta^{\frac{n}{2}-1}h=0$. Repeating the program, for some  constant $C_1\leq 0$, we have
	$$\Delta^{\frac{n}{2}-1}h=\cdots=\Delta^2 h=0, \;\Delta h\equiv C_1,\; \Delta^{\frac{n}{2}-1}\partial_jh=\cdots=\Delta\partial_jh=0.$$
	Now, apply the estimate  \eqref{Delta h+|nabla h|} to get 
	$$
	\frac{n-2}{4}	\frac{1}{|B_R(x)|}\int_{B_R(x)}|\nabla h|^2\ud y\leq o(1)-\frac{1}{|B_R(x)|}\int_{B_R(x)}\Delta h\ud y=-C_1+o(1).
	$$
	By mean value property for harmonic functions and letting $R\to\infty$, there holds
	\begin{align*}
		\partial_jh(x)=&\frac{1}{|B_R(x)|}\int_{B_R(x)}\partial_j h\ud y\\
		\leq& (\frac{1}{|B_R(x)|}\int_{B_R(x)}|\nabla h|^2\ud y)^{\frac{1}{2}}\\
		\leq& C.
	\end{align*}
	Liouville's theorem yields that $\partial_jh$ are constants and then
	$\Delta h=0.$
	Then using  \eqref{Delta h+|nabla h|} again as well as the result  $\partial_j  h$ are constants, one has 
	$$\partial_j h=0, \; 1\leq j\leq n$$
	which concludes  that
	$h$ is a constant. Finally, the solution $u$ is normal.
	Reversely, if $u$ is a normal solution, it is not hard to check that \eqref{R_g^-=o(R^n)} holds due to Lemma \ref{lem: Delta Lf}.
	
	Thus $u(x)$ is a normal solution if and only if \eqref{R_g^-=o(R^n)} holds.

\end{proof}

\begin{remark}
	Regrettably, the condition outlined in \eqref{R_g^-=o(R^n)} appears to lack a clear geometric interpretation. Nevertheless, one can utilize \eqref{R_g^-=o(R^n)} alongside H\"older's inequality to establish that $\int_{B_R(0)}(R_g^-)^{\frac{n}{2}}e^{nu}\ud x=o(R^n)$ is sufficient for the identification of normal solutions. However, it remains challenging to demonstrate whether this condition is also necessary.
\end{remark}

We will now give the proof of Theorem \ref{thm: finite volume entropy and normal metric} which provides  a different perspective to confirm the normal solution compared  with  Theorem \ref{thm: necessary and sufficient for normal solution}.

{\bf Proof of Theorem \ref{thm: finite volume entropy and normal metric}:}
\begin{proof}
	When the metric is complete and  normal, Theorem \ref{thm:d_g(x,p) and |x-p|} shows that $\tau(g)$ is finite and
	$$\tau(g)=1-\frac{2}{(n-1)!|\mathbb{S}^n|}\int_{\mr^n}Q_ge^{nu}\ud x.$$
	Conversely, if  $\tau(g)$ is finite, with help of Theorem \ref{thm:decomposition}, one has the decomposition
	$$u(x)=\mathcal{L}(Q_ge^{nu})+P(x)$$
	where $P(x)$ is a polynomial of degree at most $n-2$ and $P(x)\leq C.$ When $n=2,$  $P(x)$ is already a constant. When $n\geq 4$, we will show that $P(x)$ must be  a constant based on the completeness assumption of the metric.
	We argue by contradiction. If $P(x)$ is not a constant, since $P(x)\leq C$ and $\deg(P)\leq n-2$, there exists $k\geq1$ such that
	$P(x)=H_{2k}(x)+P_{2k-1}(x)$
	where $H_{2k}(x)$ is a non-positive  homogeneous polynomial and $P_{2k-1}(x)$ is a polynomial of degree at most $2k-1$. Then there exists $x_0\in \mr^n$ with $|x_0|=1$ and $t_0>0$ such that on the ray  $\gamma(t)=tx_0$ there holds
	$P(\gamma(t))\leq -Ct^{2k}$ for $t\geq t_0>0$. With help of the estimate \eqref{cliam for distance}, there exists $\epsilon_1>0$ and $t_2>0$ such that for any $t\geq t_2$ there holds
	$$\int^{t+1}_{t}e^{\mathcal{L}(Q_ge^{nu})(\gamma(s))}\ud s\leq t^{-\alpha_0+\epsilon_1}.$$
	Then choosing a integer  $T_1=[t_0+t_1+1]$ and intger  $T>T_1$, one has 
	\begin{align*}
		d_g(\gamma(T_1),\gamma(T))\leq& \int^T_{T_1}e^{u(\gamma(t))}\ud t\\
		\leq &\sum^{T}_{i=T_1}\int^{i+1}_{i}e^{u(\gamma(t))}\ud t\\
		\leq& \sum^{T}_{i=T_1}e^{-Ci^{2k}}\int^{i+1}_{i}e^{\mathcal{L}(Q_ge^{nu})(\gamma(t))}\ud t\\
		\leq &\sum^{T}_{i=T_1}e^{-Ci^{2k}}i^{-\alpha_0+\epsilon_1}<+\infty
	\end{align*}
	which contradicts to $d_g(\gamma(T_1),\gamma(T))\to\infty$ as $T\to\infty.$
	Thus $P(x)$  must be a constant i.e. the metric is normal.

Finally, we finish our proof.	
\end{proof}

{\bf Proof of Corollary \ref{cor}:}
\begin{proof}
	With help of the finite volume entropy assmption and Theorem \ref{thm: finite volume entropy and normal metric}, we show that the metric is normal and
	$$\tau(g)=1-\frac{2}{(n-1)!|\mathbb{S}^n|}\int_{\mr^n}Q_ge^{nu}\ud x.$$
	Since $Q_g\geq 0,$ it is obvious that
	$\tau(g)\leq 1$. If $\tau(g)=1$, it is easy to see $Q_g\equiv0$ and then $u\equiv C$ due to $u(x)$ is normal. Conversely, when $u\equiv C$, it is trivial to show  that $\tau(g)=1$.
\end{proof}

\section{Polynomial growth polyharmonic functions}\label{section: pg pf}
In this section, we will give the proof of Theorem \ref{thm: PH dimension is finite } related to polynomial growth polyharmonic functions with help of Theorem \ref{thm:d_g(x,p) and |x-p|}.

{\bf Proof of Theorem \ref{thm: PH dimension is finite }:}
\begin{proof}
	
	{\bf Case(i): }\\
	When the volume entropy is finite, Theorem \ref{thm: finite volume entropy and normal metric} shows that the metric is normal.
	With help of Theorem \ref{thm:d_g(x,p) and |x-p|}, for fixed $p\in\mr^n$ and  any $\epsilon>0$, there exists $R(\epsilon)>0$ such that for any $|x|\geq R(\epsilon)$, there holds
	\begin{equation}\label{d_g(x,p) leq |x-p|^1-alpha_0+epsilon}
		d_g(x,p)\leq |x-p|^{\tau(g)+\epsilon}.
	\end{equation}
	If $f(x)\in\mathcal{PH}_d(M,g)$, the above estimate yields that
	$$|f(x)|\leq Cd_g(x,p)^d+C\leq C|x-p|^{d(\tau(g)+\epsilon)}+C$$
	which yields that 
	$$|f(x)|\leq C|x|^{d(\tau(g)+\epsilon)}+C$$
		It is not hard  to check that
	$d(\tau(g)+\epsilon)<[d(\tau(g)+\epsilon)]+1$.
	Then we have
	$$\int_{B_R(0)}|f|\ud x=o(R^{n+[d(\tau(g)+\epsilon)]+1}).$$ 
	Due to the conformal invariant of GJMS operator,  $P_g f(x)=0$ is equivalent to
	$$(-\Delta )^{\frac{n}{2}}f(x)=0.$$
	Then with help of Lemma \ref{lem: polynomial growth is polynomial}, $f(x)$ must be a polynomial with 
	$\deg(f)\leq \max\{n-2,[d(\tau(g)+\epsilon)]\}$.
	Since the polynomial $|f(x)|\leq C|x|^{d(\tau(g)+\epsilon)}+C$,  we further  obtain
	\begin{equation}\label{deg f leq d(1-alpha_0)}
		\deg(f)\leq [d(\tau(g)+\epsilon)].
	\end{equation}
Then we have 
	\begin{equation}\label{dim <[dtau(g)+epsilon]}
		\dim(\mathcal{PH}_d(M,g))\leq\dim(\mathcal{PH}_{[d(\tau(g)+\epsilon)]}(\mr^n,|dx|^2)).
	\end{equation}
	For sufficiently small $\epsilon$, one may check
	$$[d(\tau(g)+\epsilon)]\leq d\tau(g).$$
	Thus we show that 
	$f(x)\in \mathcal{PH}_{d\tau(g)}(\mr^n,|dx|^2)$ which yields that
	\begin{equation}\label{dim _d leq dim Dtau(g)}
		\dim(\mathcal{PH}_d(M,g))\leq\dim(\mathcal{PH}_{d\tau(g)}(\mr^n,|dx|^2)).
	\end{equation}

	{\bf Case(ii): }
	
		When $\tau(g)=0$, with help of \eqref{dim _d leq dim Dtau(g)}, we have 
	$$\dim(\mathcal{PH}_d(M,g))\leq \dim(\mathcal{PH}_0(\mr^n,|dx|^2))=1.$$
	Obviuously,  since constant functions belongs to the kernal of $P_g$, we have
	$$\dim(\mathcal{PH}_d(M,g))\geq 1.$$
	Thus the identity \eqref{identiy of the dimension} holds for $\tau(g)=0.$
	
	When  $Q_g\geq 0$ near infinity, with help of Lemma \ref{lem:Q^+ and Q^-}, for $|x|\gg1$, there holds 
	\begin{equation}\label{u geq -alpha_0log|x|}
		u(x)\geq -\alpha_0\log|x|-C.
	\end{equation}
	As the arguement  in Lemma \ref{lem: inf V-g/V_0}, for $|x|\gg1$, one has
	$B_{\frac{1}{8}|x|}(\frac{x}{2})\subset B_{x,p}$ and then using \eqref{u geq -alpha_0log|x|} to get
	\begin{align*}
		\delta(x,p)^n\geq & \int_{B_{\frac{1}{8}|x|}(\frac{x}{2})}e^{nu(y)}\ud y\\
		\geq &C\int_{B_{\frac{1}{8}|x|}(\frac{x}{2})}|y|^{-n\alpha_0}\ud y\\
		\geq &C|x|^{n-n\alpha_0}\\
		=&C|x|^{n\tau(g)}.
	\end{align*}
	Then, for $\tau(g)>0$ and the estimate \eqref{strong A_infty weight property}, we have
	\begin{equation}\label{d_g geq |x|^tau(g)d}
		d_g(x,p)\geq C|x|^{\tau(g)}.
	\end{equation}
	For each $f\in \mathcal{PH}_{d\tau(g)}(\mr^n,|dx|^2)$, the estiamte \eqref{d_g geq |x|^tau(g)d} deduces that
	$$|f(x)|\leq C|x|^{d\tau(g)}+C\leq Cd(x,p)^{d}+C$$
	which shows that $f$ also belongs to $\mathcal{PH}_d(M,g)$. Thus we have
	$$\dim(\mathcal{PH}_d(M,g))\geq \dim(\mathcal{PH}_{d\tau(g)}(\mr^n,|dx|^2)).$$
	Combing with \eqref{dim _d leq dim Dtau(g)}, we obtain the identity \eqref{identiy of the dimension}.

	{\bf Case(iii): }

	When  $Q_g\geq 0$, it is easy to get   $\tau(g)\leq 1$ due to \eqref{tau (g) identity}.
	With help of \eqref{dim _d leq dim Dtau(g)} and $\tau(g)\leq 1$, it follows that for each integer $k\geq1$ $$\dim(\mathcal{PH}_k(M,g))\leq \dim(\mathcal{PH}_{k}(\mr^n,|dx|^2)).$$ If equality holds,  we claim that  $\tau(g)=1$. We argue by contradiction. If $\tau(g)<1$, due to the estimate  \eqref{deg f leq d(1-alpha_0)}, we can choose $\epsilon$ to be sufficiently small such that $\tau(g)+\epsilon<1$ and then
	$[k(\tau(g)+\epsilon)]<k$ for the integer $k\geq 1$.
	Thus
	$$\dim(\mathcal{PH}_k(M,g))\leq \dim(\mathcal{PH}_{[k(\tau(g)+\epsilon)]}(\mr^n,|dx|^2))<\dim(\mathcal{PH}_{k}(\mr^n,|dx|^2))$$
	which is a contradiction.
	When $\tau(g)=1$, we can deduce that $Q_g\equiv 0$ from \eqref{tau (g) identity} since $Q_g\geq 0$. Consequently, the  normal solution $u(x)$ must be a constant.
	Conversely, if $u(x)$ is constant, the equality obviously holds.

	{\bf Case(iv): }
	
When we have $\tau(g)=0$,
for each $d\geq 0$ and each $f\in\mathcal{PH}_d(M,g)$, the estimate \eqref{deg f leq d(1-alpha_0)} shows that
$f$ must be a constant by choosing small $\epsilon$.
Conversely, if $\mathcal{PH}_d(M,g)$ consists solely of constant functions for each $d\geq 0$, we claim that $\tau(g)=0$. We proceed by contradiction, assuming that $\tau(g)>0$. By using Theorem \ref{thm:d_g(x,p) and |x-p|} to obtain the inequality $|x-p|^{\tau(g)-\epsilon_1}\leq d_g(x,p)$ for $|x|\gg1$, where $\epsilon_1=\frac{\tau(g)}{2}$. Then, we have $|x|\leq C+Cd_g(x,p)^{\frac{2}{\tau(g)}}$. Since  $P_gx_1=0$, we obtain $x_1\in \mathcal{PH}_{\frac{2}{\tau(g)}}(M,g)$, which leads to a contradiction.

Finally, we finish our proof.	
\end{proof}

\section{Examples}\label{section: example}

Now, we are going to construct some examples to show the case $\alpha_0=1$ is very subtle by following  the examples introduced in \cite{Hu}.
Choose a smooth cut-off function
$\eta(x)$ satisfying $\eta(x)=1$ in $B_{10}(0)$ and vanishes on $\mr^n\backslash B_{20}(0)$. Consider the function
 $$w(x)=(1-\eta(x))\left(-\log|x|+c\log\log|x|\right)$$
 where $c$ is a constant.
 It is not hard to see $w(x)$ is a smooth function on $\mr^n$.
 For $|x|\geq 30$, by using polar coordinates $|x|=r$, one can check that for $|x|\geq 30$, there holds
 $$|(-\Delta)^{\frac{n}{2}}w(x)|\leq \frac{C}{r^n(\log r)^2}$$
 which shows that
$$\int_{\mr^n}|(-\Delta)^{\frac{n}{2}}w|\ud x<+\infty.$$
 Consider the metic $g=e^{2w}|dx|^2$ and 
 set the Q-curvature as 
 $$Q_g=e^{-nw}(-\Delta )^{\frac{n}{2}}w.$$
 It is not hard to check that 
 $$\int_{B_R(0)}e^{nw}\ud x= O(R),\quad \int_{B_R(0)}|w|\ud x=o(R^{n+1}).$$
 By using Theorem \ref{thm:decomposition}, we find that the metric $g$ is normal.
 Applying Lemma \ref{lem: B_1 L(f)}, we must have
 $$\int_{\mr^n}Q_ge^{nw}\ud x=\frac{(n-1)!|\mathbb{S}^n|}{2}.$$
  Choosing $x_0\in\mr^n$ with $|x_0|=1$,  for $R>30$, direct computation yields that 
 $$\int_{30}^{R}e^{u(tx_0)}\ud t=\int_{30}^Rt^{-1}(\log t)^c\ud t,\quad \int_{B_R(0)\backslash B_{30}(0)}e^{nw}\ud x=|\mathbb{S}^n|\int_{30}^Rt^{-1}(\log t)^{nc}\ud t.$$
If $c < -1$, the diameter of $(\mathbb{R}^n, e^{2w}|dx|^2)$ must be finite. For $c = 0$, one can check that the metric is complete. Similarly, when $c < -\frac{1}{n}$, the volume is finite, whereas for $c \geq  -\frac{1}{n}$, the volume is infinite.

\vspace{3em}
{\bf Acknowledgements:}  The author would like to thank Professor Xingwang Xu  for his helpful discussion and constant encouragement.     The author also wants to thank Professor Juncheng Wei, Professor Dong Ye  and Professor Mario Bonk for useful discussions. Besides, the author also wants to thank Professor Yuxin Ge for his useful discussion and hospitality in University Paul Sabatier in Toulouse.

\end{document}